\documentclass[11pt]{article}

\usepackage{amsmath}
\usepackage{verbatim}
\usepackage{amsfonts}
\usepackage{amsmath}
\usepackage{authblk}
\usepackage{amsthm}
\usepackage{amssymb}
\usepackage{graphicx,
}

\usepackage{listings}
\usepackage[bookmarksopenlevel=2,bookmarks=true]{hyperref} 

\newcommand{\E}{\mathbb{ E}}
\newcommand{\EE}{\mathbb{ E}}
\newcommand{\PP}{\mathbb{P}}

\newcommand{\R}{\mathbb{R}}
\newcommand{\C}{\mathbb{C}}

\newcommand{\HH}{\mathbb{H}}
\newcommand{\N}{\mathbb{N}}

\newcommand{\pa}{\partial}

\newcommand{\F}{{\mathcal F}}

\def\lr{\leftrightarrow}

\def\eps{\varepsilon}
\def\til{\widetilde}
\def\ha{\widehat}
\def\sem{\setminus}
\def\lin{\overline}
\def\ulin{\underline}

\def\st{\stackrel}

\DeclareMathOperator{\ee}{e} \DeclareMathOperator{\ii}{i}

\DeclareMathOperator{\rad}{rad}

\DeclareMathOperator{\sign}{sign} \DeclareMathOperator{\diam}{diam}
\DeclareMathOperator{\dist}{dist} 
\DeclareMathOperator{\hcap}{hcap} \DeclareMathOperator{\id}{id}
\DeclareMathOperator{\Imm}{Im } \DeclareMathOperator{\Ree}{Re }

\DeclareMathOperator{\doub}{doub}

\theoremstyle{plain}
\newtheorem{Theorem}{Theorem}[section]
\newtheorem{Lemma}[Theorem]{Lemma}
\newtheorem{Corollary}[Theorem]{Corollary}
\newtheorem{Proposition}[Theorem]{Proposition}
\theoremstyle{definition}
\newtheorem{Definition}[Theorem]{Definition}
\newtheorem{Remark}[Theorem]{Remark}
\newtheorem{Example}[Theorem]{Example}
\numberwithin{equation}{section}
\newcommand{\BGE}{\begin{equation}}
\newcommand{\BGEN}{\begin{equation*}}
\newcommand{\EDE}{\end{equation}}
\newcommand{\EDEN}{\end{equation*}}

\begin{document}

\title{Existence of multi-point boundary Green's function for chordal Schramm-Loewner evolution (SLE)}
\author{Rami Fakhry and Dapeng Zhan}
\affil{Michigan State University}
\date{\today}
\maketitle

\begin{abstract}
In the paper we prove that, for $\kappa\in(0,8)$, the $n$-point boundary Green's function of exponent $\frac8\kappa -1$ for chordal SLE$_\kappa$ exists. We also prove that the convergence is uniform over compact sets and the Green's function is continuous. We also give up-to-constant bounds for the Green's function.
\end{abstract}
\tableofcontents

\section{Introduction}\label{chapter1}
The Schramm-Loewner evolution (SLE for short)  is a one-parameter ($\kappa\in(0,\infty)$) family of random fractal curves which grow in plane domains. It was defined in the seminal work of Schramm \cite{S-SLE} in 1999. Because of its close relation with two-dimensional lattice models, Gaussian free field, and Liouville quantum gravity, SLE has attracted a lot of attention for over two decades.

The geometric properties of an SLE$_\kappa$ curve depends on the parameter $\kappa$. When $\kappa\ge 8$, an SLE$_\kappa$ curve visits every point in the domain (cf.\ \cite{RS}); when $\kappa\in(0,8)$, an SLE$_\kappa$ curve has Hausdorff dimension $d_0=1+\frac\kappa 8$ (cf.\ \cite{Bf}).
There are several types of SLE. In this paper we focus on chordal SLE, which grows in a simply connected plane domain from one boundary point to another boundary point.
Suppose $\gamma$ is a chordal SLE$_\kappa$ curve, $\kappa\in(0,8)$, in a domain $D$, and $z_0\in D$. The Green's function for $\gamma$ at $z_0$ is the limit
\BGE G(z_0):=\lim_{r\to 0^+} r^{-\alpha} \PP[\dist(z_0,\gamma)\le r]\label{1-pt Green}\EDE
for some suitable exponent $\alpha>0$ depending on $\kappa$ such that the limit exists and is not trivial, i.e., lies in $(0,\infty)$. This notion easily extends to $n$-point Green's function:
\BGE G(z_1,\dots,z_n):=\lim_{r_1,\dots,r_n\to 0^+} \prod_{j=1}^n r_j^{-\alpha} \PP[\dist(z_j,\gamma)\le r_j,1\le j\le n],\label{n-pt Green}\EDE
where $z_1,\dots,z_n$ are distinct points in $D$, provided that the limit exists and is not trivial.

The term ``Green's function'' is used for the following reasons. Recall the Laplacian Green's function $G_D(z,w)$, $z\ne w\in D$, for a planar domain $D$, which is characterized by the following properties: for any $w\in D$,
\begin{itemize}
  \item $G_D(\cdot,w)$ is positive and harmonic on $D\sem \{w\}$.
  \item As $z\to\pa D$, $G_D(z,w)\to 0$.
  \item As $z\to w$, $G_D(z,w)=-\frac 1{2\pi} \ln|z-w|+O(1)$.
\end{itemize}

One important fact is
\BGE  G(z,w)=  \lim_{r\to 0^+} \frac{-\ln r}{2\pi}\cdot  \PP^w[\dist(z,B[0,\tau_D])\le r],\label{Laplacian Green}\EDE
where $B$ is a planar Brownian motion started from $w$, and $\tau_D$ is the exit time of $D$. Notice the similarity between (\ref{1-pt Green}) and (\ref{Laplacian Green}). The main difference between them is the normalization factor: one is $r^{-\alpha}$ and the other is    $\frac{-\ln r}{2\pi}$.

 Another important fact is: for any measurable set $U\subset D$,
\BGE \EE^w[|\{t\in[0,\tau_D): B_t\in U\}|]=\int_U G(z,w) dA(z).\label{integral-Laplacian}\EDE
Here $|\cdot |$ stands for the Lebesgue measure on $\R$, and $A$ is the area.

It turns out that the correct exponent $\alpha$ is  the co-dimension of the SLE curve, i.e., $\alpha=2-d_0=2-(1+\frac\kappa 8)=1-\frac \kappa 8$. 
The existence of a variation of Green's function for chordal SLE$_\kappa$, $\kappa\in(0,8)$, was given in \cite{Law4}, where  the conformal radius was used instead of Euclidean distance. The existence of $2$-point Green's function was proved in \cite{LW} (again for conformal radius instead of Euclidean distance) following a method initiated by Beffara \cite{Bf}. In \cite{LR} the authors showed that Green's function as defined in (\ref{n-pt Green}) (using Euclidean distance) exists for $n = 1,2$,
and then used those Green's functions to prove that
\begin{itemize}
  \item an SLE$_\kappa$ curve $\gamma$ can be parametrized by its $d_0$-dimensional Minkowski content, i.e., for any $t_1<t_2$, the $(1+\frac\kappa 8)$-dimensional Minkowski content of $\gamma[t_1,t_2]$ is $t_2-t_1$;
  and
  \item under such parametrization, for any measurable set $U\subset D$,
\BGE \EE[|\{t: \gamma(t)\in U\}|]=\int_U G(z) dA(z).\label{integral-NP}\EDE
\end{itemize}
The Minkowski content parametrization agrees with the natural parametrization introduced earlier (cf.\ \cite{LS,LZ})
The similarity between (\ref{integral-Laplacian}) and (\ref{integral-NP}) further justifies the terminology  ``Green's function''.

In a series of papers (\cite{higher,existence}) the authors showed that the Green's function of chordal SLE exists for any $n\in\N$.  In addition, they found convergence rate and
modulus of continuity of the Green's functions, and provided up-to-constant sharp bounds for them.

If the reference point(s) is (are) on the boundary of the domain instead of the interior, we may use (\ref{1-pt Green}) and (\ref{n-pt Green}) to define the one-point and $n$-point boundary Green's function. Again, if $\kappa\ge 8$, the boundary Green's function makes no sense for SLE$_\kappa$ since it visits every point on the boundary; if $\kappa\in (0,8)$, the intersection of the SLE$_\kappa$ curve with the boundary  has Hausdorff dimension $d_1=2-\frac 8\kappa$ (\cite{dim-real}), and so the reasonable choice of the exponent $\alpha$ is $\alpha=1-d_1=\frac 8\kappa-1$.

Greg Lawler proved (cf.\ \cite{Mink-real}) the existence of the $1$- and $2$-point (on the same side) boundary Green's function for chordal SLE, and  used them to prove that the $d_1$-dimensional Minkowski content of the intersection of SLE$_\kappa$ with the domain boundary exists. He also obtained the exact formulas of these Green's function up to some multiplicative constant. We will use the exact formula of the one-point Green's function: 
\BGE G(z)=\ha c |z|^{-\alpha},\quad z\ne0,\label{G(z)}\EDE
where $\ha c>0$ is some (unknown) constant depending only on $\kappa$. We will also use the convergence rate of the one-point Green's function (\cite[Theorem 1]{Mink-real}): there is some constant $C,\beta_1>0$ depending only on $\kappa$ such that for any $z\in\R\sem\{0\}$ and $r\in (0,|z|)$,
\BGE |\PP[\dist(z,\gamma)\le \eps]-G(z)\eps^\alpha|\le C (\eps/|z|)^{\alpha+\beta_1}.\label{G(z)-approx}
\EDE


To the best of our knowledge, the existence of $n$-point boundary Green's
function for $n > 2$ and the $2$-point boundary Green's function when the two reference points lie on different sides of $0$ has not been proved so far. The main goal of the paper is to prove this existence for all $n\in\N$ without assuming that the reference points all lie on the same side of $0$. In addition we prove that the Green's functions are continuous. We do not have exact formulas of these functions, but find some sharp bounds for them in terms of simple functions.

We will mainly follow the approach in \cite{existence}, and apply the results from there as well as from \cite{Mink-real} and \cite{higher}. Below is our main result.

\begin{Theorem}
  Let $\kappa\in(0,8)$ and $\alpha=\frac 8\kappa -1$. Let $\gamma$ be an SLE$_\kappa$ curve in $\HH:=\{z\in\C:\Imm z>0\}$ from $0$ to $\infty$. Let $n\in\N$ and $\Sigma_n=\{(z_1,\dots,z_n)\in(\R\sem \{0\})^n: z_j\ne z_k\mbox{ whenever}j\ne k\}$. Then for any $\ulin z=(z_1,\dots,z_n)\in \Sigma_n$, the limit $G(\ulin z)$ in (\ref{n-pt Green})
  exists and lies in $(0,\infty)$. Moreover, the convergence in (\ref{n-pt Green}) is uniform on each compact subset of $\Sigma_n$,   the function $G$ is continuous on $\Sigma_n$, and there is an explicit function $F$ on $\Sigma_n$ (defined in (\ref{F})) with a simple form such that $G(\ulin z)\asymp F(\ulin z)$, where the implicit constants depend only on $\kappa$ and $n$. 
  \label{Main-thm}
\end{Theorem}


Our result will shed light on the study of multiple SLE. For example, if we condition the chordal SLE$_\kappa$ in Theorem \ref{Main-thm} to pass through small discs centered at $z_1<z_2<\cdots<z_n\in (0,\infty)$, and suitably take limits while sending the radii of the discs to zero, then we should get an $(n+1)$-SLE$_\kappa$ configuration in $\HH$ with link pattern $(0\lr z_1;z_1\lr z_2;z_2\lr z_3;\dots;z_{n-1}\lr z_n;z_n\lr \infty)$, which is a collection of $(n+1)$ random curves $(\gamma_0,\dots,\gamma_n)$ in $\lin\HH$ such that $\gamma_j$ connects $z_j$ with $z_{j+1}$, where $z_0:=0$ and $z_{n+1}:=\infty$, and when any $n$ curves among the $(n+1)$ curves are given, the last curve is a chordal SLE$_\kappa$ curve in a connected component of the complement of the given $n$ curves in $\HH$.  The $n$-point boundary Green's function is then closely related to the partition function associated to such multiple SLE.

Here are a few topics that we could study in the near future. We may consider ``mixed'' multi-point Green's functions for chordal SLE, where some reference points lie in the interior of the domain, and some lie on the boundary. We expect that the Green's functions still exist, if  $1-\frac\kappa 8$ is used as the exponent for interior points and   $\frac 8\kappa-1$ is used as the exponent for boundary points. We may also work on other types of SLE such as radial SLE, which grows from a boundary point to an interior point.  The multi-point (interior) Green's function for radial SLE was proved to exist in \cite{MZ}. The next natural objects to study are the boundary and mixed multi-point Green's function for radial SLE.

The rest of the paper is organized in a straightforward fashion. In Section \ref{Chap-Prel}, we recall symbols, notation and some basic results that are relevant to the paper. Section   \ref{mainestimates} contains the most technical part of the paper, where we derive a number of important estimates. We finish the proof of the main theorem in Section \ref{Chap-main-thm}.

\section{Preliminaries}\label{Chap-Prel}
\subsection{Notion and symbols}
Let $\HH=\{z\in\C:\Imm z>0\}$ be the open upper half plane. Given $z_0\in\C$ and $S\subset \C$, we use $\rad_{z_0}(S)$ to denote $\sup\{|z-z_0|:z\in S\cup\{z_0\}\}$. We write $\N_n$ for $\{k\in\N:k\le n\}$, where $\N=\{1,2,3,\dots\}$ is the set of all positive integers. For $a,b\in\R$, we write $a\wedge b$ and $a\vee b$ respectively for $\min\{a,b\}$ and $\max\{a,b\}$.

We fix $\kappa\in(0,8)$ and set $d=1+\frac\kappa 8$ and $\alpha=\frac 8\kappa -1$. Throughout, a constant (such as $\alpha$) depends only on $\kappa$ and a variable $n\in\N$ (number of points), unless otherwise specified. We use $X\lesssim Y$ or $Y\gtrsim X$ if there is a constant $C>0$ such that $X\le C Y$. We write $X\asymp Y$ if $X\lesssim Y$ and $Y\lesssim X$. 



When a (deterministic or random) curve $\gamma(t)$, $t\ge 0$, is fixed in the context, we let $\tau_S=\inf(\{t\ge 0:\gamma(t)\in S\}\cup\{\infty\})$. We write $\tau^{z_0}_r$ for $\tau_{\{z:|z-z_0\le r\}}$, and $T_{z_0}$ for $\tau^{z_0}_0=\tau_{\{z_0\}}$. So another way to say that $\dist(z_0,\gamma)\le r$ is $\tau^{z_0}_r<\infty$. 
We also write $\tau^\infty_R$ for $\tau_{\{z:|z|\ge R\}}$.

A crosscut in a domain $D$ is an open simple curve in $D$, whose two ends approach to two boundary points of $D$. When $D$ is a simply connected domain, any crosscut $\rho$ of $D$ divides $D$ into two connected components.

\subsection{$\HH$-Hulls}\label{H-hull}
 A relatively closed bounded subset $K$ of $\HH$ is called an $\HH$-hull if $\HH\sem K$ is simply connected. The complement domain $\HH\sem K$ is then called an $\HH$-domain. Given an $\HH$-hull $K$, we use $g_K$ to denote the unique conformal map from $\HH\sem K$ onto $\HH$ that satisfies $g_K(z)=z+O(|z|^{-1})$ as $z\to \infty$. Let $f_K=g_K^{-1}$. The half-plane capacity of $K$ is $\hcap(K):=\lim_{z\to\infty} z(g_K(z)-z)$. If $K=\emptyset$, then $g_K=f_K=\id$, and $\hcap(K)=0$. Now suppose $K\ne\emptyset$. Let $a_K=\min(\lin K\cap\R)$ and $b_K=\max(\lin K\cap\R)$. Let $K^{\doub}=K\cup[a_K,b_K]\cup\{\lin z: z\in K\}$. By Schwarz reflection principle, $g_K$ extends to a conformal map from $\C\sem K^{\doub}$ onto $\C\sem[c_K,d_K]$ for some $c_K<d_K\in\R$, and satisfies $g_K(\lin z)=\lin{g_K(z)}$. In this paper, we write $S_K$ for $[c_K,d_K]$. In the case $K=\emptyset$, we understand $S_K$ and $[a_K,b_K]$ as the empty set. Given two $\HH$-hulls $K_1\subset K_2$, we get another $\HH$-hull $K_2/K_1$ defined by $K_2/K_1=g_{K_1}(K_2\sem K_1)$.

\begin{Example}
 For $x_0\in\R$ and $r>0$,  the set $K:=\{z\in\HH:|z-x_0|\le r\}$  is an $\HH$-hull, $a_K=x_0-r$, $b_K=x_0+r$, $g_K(z)=z+\frac{r^2}{z-x_0}$, $\hcap( K)=r^2$, and $S_{K}=[x_0-2r,x_0+2r]$. \label{semi-disc}
\end{Example}

 \begin{Proposition}
 For any $\HH$-hull $K$, $[a_K,b_K]\subset S_K$.
   If $K_1\subset K_2$ are two $\HH$-hulls, then  $S_{K_1}\subset S_{K_2}$ and $S_{K_2/K_1}\subset S_{K_2}$. \label{SKSH}
 \end{Proposition}
 \begin{proof}
 This is \cite[Lemmas 5.2 and 5.3]{LERW}.
 \end{proof}


\begin{Proposition}
	If a nonempty $\HH$-hull $K$ satisfies that $\rad_{x_0}(K)\le r$ for some $x_0\in\R$ and $r>0$, then $\hcap(K)\le r^2$, $S_K\subset[x_0-2r,x_0+2r]$, and   \BGE |g_K(z)-z|\le 3r, \quad z\in\C\sem K^{\doub}.\label{f-z0}\EDE
Moreover, for any $z\in\C$ with $|z-x_0|\ge 5r$, we have
	\BGE |g_K(z)-z|\le 2|z-x_0|\Big(\frac{r}{|z-x_0|}\Big)^2 ;\label{1}\EDE
	\BGE |g_K'(z)-1|\le 5\Big(\frac{r}{|z-x_0|}\Big)^2 .\label{3}\EDE
\label{small}
\end{Proposition}
\begin{proof}
This is  \cite[Lemmas 2.5 and 2.6]{existence}.
\end{proof}


\begin{Proposition}
Let $H$ be a nonempty $\HH$-hull, and ${\mathcal H}(H)$ denote the space of  $\HH$-hulls, which are subsets of $H$. Then ${\mathcal H}(H)$ is compact in the sense that any sequence $(K_n)$ in ${\mathcal H}(H)$ contains a convergent subsequence $(K_{n_k})$ whose limit $K$ is contained in ${\mathcal H}(H)$.  Here the convergence means that $g_{K_{n_k}}$ converges to $g_K$ locally uniformly in $\C\sem H^{\doub}$.
\label{compact-prop}
\end{Proposition}
\begin{proof}
  This is \cite[Lemma 5.4]{LERW}.
\end{proof}

\subsection{Chordal Loewner Processes}
Let $U(t)$, $0\le t<T$, be a real valued continuous function, where $T\in(0,\infty]$. The chordal Loewner equation driven by $U$ is the equation \BGE \pa_t g_t(z)=\frac 2{g_t(z)-U_t},\quad g_0(z)=z.\label{chordal}\EDE
For every $z\in\C$, let $\tau^*_z$ denote the first time that the solution $g_\cdot(z)$ blows up; when such time does not exist, $\tau^*_z$ is set to be $\infty$. Let $K_t=\{z\in\HH:\tau^*_z\le t\}$. We call $g_t$ and $K_t$, $0\le t<T$, the chordal Loewner maps and hulls, respectively, driven by $U$. It turns out that, for  each $t\in[0,T)$, $K_t$ is an $\HH$-hull, $\hcap(K_t)=2t$, and $g_t=g_{K_t}$.

\begin{Proposition}
  For any $0\le t<T$,
  $$\{U_t\}=\bigcap_{\eps\in(0,T-t)} \lin{K_{t+\eps}/K_t}.$$ \label{K/U}
\end{Proposition}
\begin{proof}
 This a restatement of \cite[Theorem 2.6]{LSW1}.
\end{proof}

\begin{Corollary}
If for some $\HH$-hull $H$ and $t_0\in(0,T)$, $K_{t_0}\subset H$,  then $U_t\in S_H$ for $0\le t<t_0$. \label{UinSK}
\end{Corollary}
\begin{proof}
 By Proposition \ref{K/U}, for every $t\in[0,t_0)$, $U_t\in [a_{K_{t_0}/K_t},b_{K_{t_0}/K_t}]$, which implies by Proposition \ref{SKSH} that $U_t\in S_{K_{t_0}/K_t}\subset S_{K_{t_0}}\subset S_H$. By the continuity of $U$, we also have $U_{t_0}\in S_H$.
\end{proof}

We call the maps $Z_t=g_t-U_t$ the centered Loewner maps driven by $U$.

\begin{Proposition}
 Let $b>a\in [0,T)$. Suppose that $\rad_{x_0}(K_b/K_a)\le r$ for some $x_0\in\R$ and $r>0$. Then  $|Z_{a}(z)-Z_b(z)|\le 7r$ for any $z\in\lin\HH\sem \lin{K_b}$.
\label{centered-Delta}
\end{Proposition}
\begin{proof}
Let $U_{a;t}=U_{a+t}$, $g_{a;t}=g_{a+t}\circ g_a^{-1}$, and $K_{a;t}=K_{a+t}/K_a$, $0\le t<T-a$. It is straightforward to check that $g_{a;\cdot}$ and $K_{a;\cdot}$ are respectively the chordal Loewner maps and hulls driven by $U_{a;\cdot}$. By Corollary \ref{UinSK}, $U_a,U_b\in S_{K_{a;b-a}}$. By the assumption, $\rad_{x_0 }(K_{a;b-a})\le r$. By Proposition \ref{small}, $S_{K_{a;b-a}}\subset [x_0 -2r,x_0 +2r]$. Thus, $|U_a-U_b|\le 4r$.
By Proposition \ref{small}, $|g_{a;b-a}(z)-z|\le 3r$ for any $z\in \lin\HH\sem \lin{K_{a;b-a}}$. So for any $z\in\lin\HH\sem \lin{K_b}$, $|g_a(z)-g_b(z)|\le 3 r$. Since $Z_t=g_t-U_t$, and $|U_a-U_b|\le 4r$, we get the conclusion.
\end{proof}

If there exists a function $\gamma(t)$, $0\le t<T$, in $\lin\HH$, such that for any $t$, $\HH\sem K_t$ is the unbounded connected component of $\HH\sem \gamma[0,t]$, we say that such $\gamma$ is the chordal Loewner curve driven by $U$. Such $\gamma$ may not exist in general, but when it exists, it is determined by $U$, and for each $t\in[0,T)$, $g_t^{-1}$ and $Z_t^{-1}$ extend continuously from $\HH$ to $\lin\HH$ and satisfy $g_t^{-1}(U_t)=Z_t^{-1}(0)=\gamma(t)$.


\subsection{Chordal SLE}
Let $\kappa>0$. Let $B_t$ be a standard Brownian motion. If the driving function is $U_t=\sqrt\kappa B_t$, $0\le t<\infty$, then the chordal Loewner curve driven by $U$ exists,  starts from $0$ and ends at $\infty$  (cf.\ \cite{RS}). Such curve is called a chordal SLE$_\kappa$ trace or curve in $\HH$ from $0$ to $\infty$. Its geometric property depends on $\kappa$: if $\kappa\le 4$, it is simple; if $4<\kappa<8$, it is not simple and not space-filling; if $\kappa\ge 8$, it is space-filling (cf.\ \cite{RS}). The Hausdorff dimension of an SLE$_\kappa$ curve is $\min\{1+\frac\kappa 8,2\}$ (cf. \cite{RS,Bf}).

The definition of chordal SLE extends to general simply connected domains via conformal maps. Let $D$ be a simply connected domain with two distinct boundary points (more precisely, prime ends) $a,b$. Let $f$ be a conformal map from $\HH$ onto $D$, which sends $0$ and $\infty$ respectively to $a$ and $b$. Let $\gamma$ be a chordal SLE$_\kappa$ curve in $\HH$ from $0$ to $\infty$. Then $f\circ \gamma$ is called a chordal SLE$_\kappa$ curve in $D$ from $a$ to $b$.

A remarkable property of SLE is the Domain Markov Property (DMP). Suppose $\gamma$ is a chordal SLE$_\kappa$ curve in $\HH$ from $0$ to $\infty$, which generates the $\HH$-hulls $K_t$, $0\le t<\infty$, and a filtration $\F=(\F_t)_{t\ge 0}$. Let $\tau$ be a finite $\F$-stopping time. Conditionally on $\F_\tau$,
 $\gamma(\tau+\cdot)$ has the same law of a chordal SLE$_\kappa$ curve in $\HH\sem K_\tau$ from $\gamma(\tau)$ to $\infty$. Equivalently, there is a chordal SLE$_\kappa$ curve $\til\gamma$ in $\HH$ from $0$ to $\infty$ independent of $\F_\tau$ such that $\gamma(\tau+t)=Z_\tau^{-1}(\til \gamma(t))$, $t\ge 0$.
Here $Z_\tau$ is the centered Loewner map at the time $\tau$ that corresponds to $\gamma$, and its inverse $Z_\tau^{-1}$ has been extended continuously to $\lin\HH$.

We will also use the left-right symmetry and rescaling property of chordal SLE. Suppose $\gamma$ is a chordal SLE$_\kappa$ curve in $\HH$ from $0$ to $\infty$. The left-right symmetry states that, if $f(z)=-\lin z$ is the reflection about $i\R$, then $f\circ \gamma$ has the same law as $\gamma$. This follows easily from that $(-\sqrt \kappa B_t)$ has the same law as $(\sqrt\kappa B_t)$. The rescaling property states that, for any $c>0$, $(c\gamma(t))$ has the same law as $(\gamma(\sqrt c t))$. This follows easily from the rescaling property of the Brownian motion.

\subsection{Extremal Length}\label{Extremal}
We will need some lemmas on extremal length, which is a nonnegative quantity $\lambda(\Gamma)$ associated with a family $\Gamma$ of rectifiable curves (\cite[Definition 4-1]{Ahl}). One remarkable property of extremal length is its conformal invariance (\cite[Section 4-1]{Ahl}), i.e., if every $\gamma\in\Gamma$ is contained in a domain $\Omega$, and $f$ is a conformal map defined on $\Omega$, then $\lambda(f(\Gamma))=\lambda(\Gamma)$.
We use $d_\Omega(X,Y)$ to denote the extremal distance between $X$ and $Y$ in $\Omega$, i.e., the extremal length of the family of curves in $\Omega$ that connect $X$ with $Y$. It is known that in the special case when $\Omega$ is a semi-annulus $\{z\in\HH:R_1<|z-x|<R_2\}$, where $x\in\R$ and $R_1>R_1>0$, and $X$ and $Y$ are the two boundary arcs $\{z\in\HH:|z-x|=R_j\}$, $j=1,2$, then $d_\Omega(X,Y)=\log(R_2/R_1)/\pi$ (\cite[Section 4-2]{Ahl}). We will use the comparison principle (\cite[Theorem 4-1]{Ahl}): if every $\gamma\in\Gamma$ contains a $\gamma'\in\Gamma'$, then $\lambda(\Gamma)\ge \lambda(\Gamma')$. Thus, if every curve in $\Omega$ connecting $X$ with $Y$  a semi-annulus with radii $R_1,R_2$, then $d_\Omega(X,Y)\ge \log(R_2/R_1)/\pi$. We will also use the composition law (\cite[Theorem 4-2]{Ahl}): if for $j=1,2$, every $\gamma_j$ in a family $\Gamma_j$ is contained in $\Omega_j$, where $\Omega_1$ and $\Omega_2$ are disjoint open sets, and if every $\gamma$ in another family $\Gamma$ contains a $\gamma_1\in\Gamma_1$ and a $\gamma_2\in\Gamma_2$, then $\lambda(\Gamma)\ge \lambda(\Gamma_1)+\lambda(\Gamma_2)$.

The following propositions are applications of Teichm\"uller Theorem.

\begin{Proposition}
	Let $S_1$ and $S_2$ be a disjoint pair of connected closed subsets of $\lin\HH$ that intersect $\R$ such that $S_1$ is bounded and $S_2$ is unbounded. Let $z_j\in S_j\cap\R$, $j=1,2$. Then
$$1\wedge \frac{\rad_{z_1}(S_1)}{|z_1-z_2|}\le 32 e^{-\pi d_{\HH}(S_1,S_2)}.$$
\label{lem-extremal2}
\end{Proposition}
\begin{proof}
  For $j=1,2$, let $S_j^{\doub}$ be the union of $S_j$ and its reflection about $\R$. By reflection
principle (\cite[Exercise 4-1]{Ahl}), $d_{\HH}(S_1,S_2)=2d_{\C}(S_1^{\doub},S_2^{\doub})$.  Let $r=\rad_{z_2}(S_1)$, $L=|z_1-z_2|$ and $R=L/r$. From Teichm\"uller Theorem (\cite[Theorem 4-7]{Ahl}),
$$ d_{\C}(S_1^{\doub},S_2^{\doub})\le d_{\C}([-r,0],[L,\infty))= d_{\C}([-1,0],[R,\infty))=\Lambda(R).$$
From \cite[Formula (4-21)]{Ahl}, we have
$$e^{-\pi d_{\HH}(S_1,S_2)}=e^{-2\pi d_{\C}(S_1^{\doub},S_2^{\doub})}\ge e^{-2\pi \Lambda(R)}\ge \frac 1{16(R+1)} .$$
Since $1\wedge \frac 1 R\le \frac 2{1+R}$, we get the conclusion.
\end{proof}

\begin{Proposition}
  Let $D$ be an $\HH$-domain and $S\subset \HH$. Suppose that there are $z_0\in\R$ and $r>0$ such that $\{|z-z_0|=r\}\cap D$ has a connected component $C_r$, which disconnect $S$ from $\infty$. In other words, $S$ lies in the bounded component of $D\sem C_r$. Let $g$ be a conformal map from $D$ onto $\HH$ such that $\lim_{z\to \infty} g(z)/z=1$. Then there is $w_0\in\R$ such that
  $$\rad_{w_0}(g(S))\le 4 \rad_{z_0}(S).$$ \label{lem-extremal2'}
\end{Proposition}
\begin{proof}
Since $C_r$ is a crosscut of $D$, and $S$ lies in the bounded component of $D\sem C_r$, $g(C_r)$ is a crosscut of $g(D)=\HH$, and $g(S)$ lies in the bounded component of $\HH\sem g(C_r)$. Let $w_0$ be one endpoint of $g(C_r)$. It suffices to show that $\rad_{w_0}(g(C_r))\le 4 r$.   Let $L=\rad_{z_0}(K)$ and $L_0=|z_0-w_0|$. Take a big number $R>r+L+L'$ and let $C_R=\{z\in\HH:|z-z_0|=R\}$. Then $S$ and $C_R$ can be separated by the semi-annulus $\{z\in\HH:r<|z-x_0|<R\}$ in $D$. By the comparison principle and conformal invariance of extremal length,
  $$d_{\HH}(g(C_r),g(C_R))=d_{D}(C_r,C_R)\ge \frac 1\pi \log(R/r).$$
  Let $K$ be the $\HH$-hull $\HH\sem D$. Then $g-g_K$ is a real constant. So we may assume that $g=g_K$.

By Proposition \ref{small} again, $g(C)$ is a crosscut of $\HH$ with  $\rad_{w_0}(g(C))\le R+L+L_0$.
 Let $r'=\rad_{w_0}(g(C_r))$ and $R'=R+L+L_0$. By comparison principle, reflection principle and Teichm\"uller Theorem,
 $$d_{\HH}(g(C_r),g(C_R))\le d_{\HH}(g(C_r),\{z\in\HH:|z-w_0|=R'\})$$
 $$=2 d_{\C}(g(C_r)^{\doub},\{|z-w_0|=R'\})\le 2 d_{\C}([-1,0], \{|z|=R'/{r'}\}=2M(R'/r').$$
 By \cite[Formula 4-14]{Ahl}, $2M(R'/r')= \Lambda((R'/r')^2-1)$. Thus, by the above displayed formulas and \cite[Formula (4-21)]{Ahl}
 $$\frac 1\pi \log(R/r)\le \Lambda((R'/r')^2-1)\le \frac 1{2\pi} \log (16(R'/r')^2)=\frac 1\pi (4(R'/r')).$$
 So we get $r'\le  4(R'/R) r$. Letting $R\to \infty$, we get $R'/R\to 1$. So $r'\le 4r$.
\end{proof}

\subsection{Two-sided Chordal SLE}
Suppose $\gamma$  is a chordal SLE$_\kappa$ curve in $\HH$ from $0$ to $\infty$, which generates the filtration $\F=(\F_t)_{t\ge 0}$. Let $\PP$ denote the law of $\gamma$, and $\EE$ denote the
corresponding expectation.
Let $z\in\R\sem \{0\}$. By (\ref{chordal}) and the fact that $U_t=\sqrt\kappa B_t$ for some standard Brownian motion $B_t$, up to $\tau^*_z$, $Z_t(z)$ and $g_t'(z)$ satisfy the following SDE and ODE:
\begin{align*}
 d Z_t(z)&=-\sqrt\kappa dB_t+\frac 2{Z_t(z)} \,dt;\\
\frac{ d g_t'(z)}{g_t'(z)}&=\frac{-2}{Z_t(z)^2}\,dt.
\end{align*}
By It\^o's formula (cf.\ \cite{RY}), we get the following continuous positive local martingale:
\BGE M_t(z):=\frac{|g_t'(z)|^\alpha |z|^\alpha}{|Z_t(z)|^\alpha}, \quad 0\le t<\tau^*_z,\label{M}\EDE
which satisfies the SDE:
\BGE \frac{dM_t(z)}{M_t(z)}=\frac{\kappa-8}{\sqrt\kappa} \frac{dB_t}{Z_t(z)}.\label{dM}\EDE
By Girsanov Theorem (cf.\ \cite{RY}), if we tilt the law $\PP$ by the local martingale $M_\cdot(z)$, we get a new random curve $\til\gamma$, whose driving function $\til U$ satisfies the SDE:
$$d\til U_t=\sqrt\kappa d\til B_t +\frac{\kappa-8}{\til Z_t(z)}\,dt,$$
where $\til B$ is another standard Brownian motion, and $\til Z_t$'s are the centered Loewner maps associated with $\til\gamma$. In fact, such $\til \gamma$ is a chordal SLE$_\kappa(\kappa-8)$ curve (cf.\ \cite{LSW-8/3}) in $\HH$ started from $0$, aimed at $\infty$, with the force point located at $z$. Since $\kappa-8<\frac\kappa 2-4$, with probability $1$, $\til\gamma$ ends at $z$ (cf.\ \cite{MS1}).

The above curve $\til\gamma$ from $0$ to $z$ is the first arm of a  two-sided chordal SLE$_\kappa$ curve in $\HH$ from $0$ to $\infty$ passing through $z$. Given this arm $\til\gamma$, the rest of the two-sided chordal SLE$_\kappa$ curve is a chordal SLE$_\kappa$ curve from $z$ to $\infty$   in the unbounded connected component of  $\HH\sem\til\gamma$. We use $\PP^*_z$ to denote the law of such a two-sided chordal SLE$_\kappa$ curve, and let  $\EE^*_z$ denote the corresponding expectation.
For $r>0$, we use $\PP^r_z$ to denote the conditional law $\PP[\cdot|\tau^z_r<\infty]$, i.e., the law of a chordal SLE$_\kappa$ curve in $\HH$ from $0$ to $\infty$ conditioned to visit the disk with radius $r$ centered at $z$; and let $\EE^r_z$ denote the corresponding expectation.

\begin{Proposition}
	Let $z\in\R\sem \{0\}$ and $R\in(0,|z|)$. Then $\PP_z^*$ is absolutely continuous w.r.t.\ $\PP_z^R$ on $\F_{\tau^z_R}\cap\{\tau^z_R<\infty\}$, and the Radon-Nikodym derivative is uniformly bounded by some constant $C_\kappa\in[1,\infty)$ depending only on $\kappa$. \label{RN<1}
\end{Proposition}
\begin{proof}
By symmetry we may assume $z>0$.  Let $\tau=\tau^z_R$.
By the construction of $\PP_z^*$ (through tilting $\PP$ by $M_\cdot(z)$), we have
$$ \frac{d\PP_z^*|{\F_{\tau }\cap\{\tau <\infty\}}}{d\PP |{\F_{\tau }\cap\{\tau <\infty\}}}= {M_\tau(z)} .$$ 
By the definition of $\PP_z^R$,
$$ \frac{d\PP_z^R|{\F_{\tau }\cap\{\tau <\infty\}}}{d\PP |{\F_{\tau }\cap\{\tau <\infty\}}}=\frac{1}{\PP[\tau <\infty]}.$$ 
Thus, it suffices to prove that $M_\tau(z ) \cdot \PP[\tau <\infty]$ is uniformly bounded. By (\ref{1pt}), $\PP[\tau <\infty]\lesssim (R/|z|)^\alpha$. Since $g_\tau$ maps the simply connected domain $\Omega:=\C\sem (K_\tau^{\doub}\cup (-\infty,0])$ conformally onto $\C\sem (-\infty, b_{K_\tau}]$, by Koebe's $1/4$ theorem,
\BGE |g_\tau'(z)|\cdot R=|g_\tau'(z)|\cdot \dist(z,\pa \Omega)\asymp \dist(g_\tau(z),\pa(\C\sem  (-\infty, b_{K_\tau}]))=|g_\tau(z)-b_{K_\tau}|\le Z_t(z),\label{gz-Z}\EDE
where in the last step we used $g_\tau(z)>b_{K_\tau}\ge U_\tau$. Thus,
$$M_\tau(z)\cdot \PP[\tau <\infty]\lesssim \frac{|g_t'(z)|^\alpha |z|^\alpha}{|Z_t(z)|^\alpha}\frac{R^\alpha}{|z|^\alpha} =  \frac{|g_t'(z)|^\alpha R|^\alpha}{|Z_t(z)|^\alpha}\lesssim 1.$$
\end{proof}

\begin{Proposition}
  Let $z\in\R\sem \{0\}$ and $0<r<\eta<|z|$. Then $\PP^r_z$ restricted to $\F_{\tau_\eta^z}$ is absolutely continuous with respect to $\PP^*_z$, and there is a constant $ \beta>0$ depending only on $\kappa$ such that
\BGE \Big|\log\Big(\frac{d\PP^r_z|\F_{\tau_\eta^z}}{d\PP^*_z|\F_{\tau_\eta^z}}\Big) \Big| \lesssim \Big(\frac r \eta\Big)^{\beta},\quad \mbox{if }r/\eta<1/6.
\label{Prop2.13-eqn}\EDE
  \label{Prop2.13}
\end{Proposition}
\begin{proof}
Recall the $G(z)=\ha c |z|^{-\alpha}$ defined by (\ref{G(z)}). Define $G_t(z)=|Z_t'(z)|^\alpha Z_t(z)$ if $\tau^*_z>t$; and $G_t(z)=0$ if $\tau^*_z\le t$. Then
$$\frac{d\PP^*_z|\F_{\tau_\eta^z}}{d\PP|\F_{\tau_\eta^z}}=M_{\tau^z_\eta}(z)=\frac{G_{\tau^z_\eta}(z)}{G(z)}.$$
By the definition of $\PP^r_z$, we have
$$\frac{d\PP^r_z|\F_{\tau_\eta^z}}{d\PP|\F_{\tau_\eta^z}}=\frac{\PP[\tau^z_r<\infty|\F_{\tau^\eta_z}]}{\PP[\tau^z_r<\infty]}.$$
Since $\PP[\tau^z_r<\infty|\F_{\tau^\eta_z}]=0$ implies that $\tau^\eta_z\le \tau^*_z$,  which in turn implies that $G_{\tau^z_\eta}(z)=0$, by the above two displayed formulas, $\PP^r_z$ restricted to $\F_{\tau_\eta^z}$ is absolutely continuous with respect to $\PP^*_z$, and
$$\frac{d\PP^r_z|\F_{\tau_\eta^z}}{d\PP^*_z|\F_{\tau_\eta^z}} =\frac{\PP[\tau^z_r<\infty|\F_{\tau^\eta_z}]/(G_{\tau^z_\eta}(z)r^\alpha)}{\PP[\tau^z_r<\infty]/(G(z)r^\alpha)}.$$
By (\ref{G(z)-approx}) and Koebe's distortion theorem, there are constants $\beta,\delta>0$ such that, if $r/\eta<1/6$, then
$$\log(\PP[\tau^z_r<\infty]/(G(z)r^\alpha))\lesssim (r/|z|)^{\beta},\quad \log(\PP[\tau^z_r<\infty|\F_{\tau^\eta_z}]/(G_{\tau^z_\eta}(z)r^\alpha))\lesssim (r/\eta)^{\beta}.$$
The above two displayed formulas together imply (\ref{Prop2.13-eqn}).
\end{proof}

  \section{Main Estimates} \label{mainestimates}
In this section, we will provide some useful estimates for the proof of the main theorem. We use the notion and symbols in the previous section.
We now define the function $F(z_1,\dots,z_n)$ that appeared in Theorem \ref{Main-thm}.


From now on, let $d_0=1+\frac \kappa 8$ and $\alpha=\frac 8\kappa -1$.
For $y\ge 0$, define $P_y$ on $[0,\infty)$ by
$$P_y(x)=\left\{
\begin{array}{ll}
y^{\alpha-(2-d_0)} x^{2-d_0},&x\le y;\\
x^\alpha,& x\ge y.
\end{array}
\right.
$$

For an (ordered) set of distinct points $z_1,\dots,z_n\in\lin\HH\sem \{0\}$, we let $z_0=0$ and define
\BGE y_k=\Imm z_k,\quad l_k=\min_{0\le j\le k-1}\{|z_k-z_j|\},\quad R_k=\min_{0\le j\le n,j\ne k}\{|z_k-z_j|\},\quad 1\le k\le n.\label{lR}\EDE
Note that we have $R_k\le l_k$. 
For $r_1,\dots,r_n>0$, define
\BGE F(z_1,\dots,z_n;r_1,\dots,r_n)=\prod_{k=1}^n \frac{P_{y_k}(r_k)}{P_{y_k}(l_k)}.\label{Fzr}\EDE
The following is \cite[Formula (2.7)]{existence}. For any permutation $\sigma$ of $\{1,\dots,n\}$,
\BGE F(z_1,\dots,z_n;r_1,\dots,r_n)\asymp F(z_{\sigma(1)},\dots,z_{\sigma(n)};r_{\sigma(1)},\dots,r_{\sigma(n)}).\label{perm-Fzr}\EDE

The following proposition combines \cite[Theorem 1.1]{higher} (which gives the upper bound) and \cite[Theorem 4.3]{existence} (which gives the lower bound).

\begin{Proposition}
Let $z_1,\dots,z_n$ be distinct points on $\lin\HH\sem \{0\}$. Let $R_1,\dots,R_n$ be defined by (\ref{lR}). Let $r_j>0$, $1\le j\le n$. Then for a chordal SLE$_\kappa$ curve $\gamma$ in $\HH$ from $0$ to $\infty$, we have
\begin{itemize}
  \item $\PP[\tau^{z_j}_{r_j}<\infty,1\le j\le n]\lesssim \prod_{j=1}^n (1\wedge  \frac{P_{y_j}(  r_j)}{P_{y_j}(l_j)} )$;
  \item $\PP[\tau^{z_j}_{r_j}<\infty,1\le j\le n]\gtrsim F(z_1,\dots,z_n;r_1,\dots r_n)$, if $r_j\le R_j$, $1\le j\le n$.
\end{itemize}
\label{lower-upper}
\end{Proposition}

Now suppose $z_1,\dots,z_n$ are distinct points on $\R\sem \{0\}$. Then $y_k=0$, $1\le k\le n$. So, the $\frac{P_{y_k}(r_k)}{P_{y_k}(l_k)}$ in (\ref{Fzr}) simplifies to $\frac{r_k^\alpha}{l_k^\alpha}$. Then we define
\BGE F(z_1,\dots,z_n)= \prod_{k=1}^n r_k^{-\alpha} F(z_1,\dots,z_n;r_1,\dots,r_n) =\prod_{k=1}^n l_k^{-\alpha}. \label{F}\EDE
This function is different from the $F(z_1,\dots,z_n)$ that appeared in \cite{existence}, which was defined for $z_1,\dots,z_n\in\HH$. By (\ref{perm-Fzr}), we have
\BGE F(z_1,\dots,z_n )\asymp F(z_{\sigma(1)},\dots,z_{\sigma(n)} ).\label{perm-F}\EDE

A simple but useful special case of Proposition \ref{lower-upper} is: when $n=1$ and $z_1\in\R\sem\{0\}$, we have 
\BGE \PP[\tau^{z_1}_{r_1}<\infty]\asymp (r_1/|z_1|)^\alpha,\quad 0<r<|z_1|.\label{1pt}\EDE
The estimate includes a lower bound and an upper bound. They  first appeared in \cite{boundary}.
The upper bound in (\ref{1pt}) was called the boundary estimate in the literature.

From now on till the end of this section, $\PP$ denotes the law of a chordal SLE$_\kappa$ curve in $\HH$ from $0$ to $\infty$; for $z\in\R\sem\{0\}$ and $r>0$, $\PP_z^r$ denotes the conditional law $\PP[\cdot|\tau^z_r<\infty]$, and $\PP_z^*$ denotes the law of a two-sided chordal SLE$_\kappa$ curve in $\HH$ from $0$ to $\infty$ passing through $z$. When $\gamma$ follows some law above in the context, let $U_t$, $K_t$ and $g_t$ be respectively the chordal Loewner driving function, hulls and maps which correspond to $\gamma$. Let $Z_t=g_t-U_t$ be the centered Loewner maps, and let $H_t=\HH\sem K_t$.   For $t\ge 0$, let $S^+_{t}$ be the set of prime ends of $H_t$ that lie on the right side of $\gamma[0,t]$ or on $[b_{K_t},\infty)$, and   let $S^-_{t}$   be the set of prime ends of $H_t$ that lie on the left side of $\gamma[0,t]$ or on $(-\infty,a_{K_t}]$.
 More precisely,  $S^+_t$ and $S^-_t$ are respectively the images of  $[0,+\infty)$ and $(-\infty,0]$ under $Z_t^{-1}$. 

\begin{Proposition}
	Let $z_1,\dots,z_{n}$ be distinct points in $\R\sem\{0\}$, where $n\ge 2$. Let $R_1,\dots,R_n$ be defined by (\ref{lR}). Let $r_j\in(0,R_j/8)$, $1\le j\le n$. 
	Then we have a constant $\beta>0$ such that for any $k_0\in\{2,\dots,n\}$ and $s_{0}\ge 0$,
\begin{align*}
  &\PP[\tau^{z_1}_{r_1}< \tau^{z_k}_{r_k}<\infty,2\le k\le n; \dist(z_{k_0},\gamma[0,\tau^{z_1}_{r_1}])\le s_{0} ]\\
  \lesssim & F(z_1,\dots,z_{n} )\prod_{j=1}^n r_j^\alpha \Big(\frac{s_{0}}{|z_{k_0}-z_1|\wedge |z_{k_0}|}\Big)^{\beta}.
\end{align*}  \label{RZ-Thm3.1}
\end{Proposition}
\begin{proof}
  This proposition is very similar to  \cite[Theorem 3.1]{existence}. 	
  The following estimate is \cite[Formula (A.14)]{existence}. For distinct points $z_1,\dots,z_n\in\lin\HH\sem\{0\}$,    $r_j\in(0,R_j)$, $1\le j\le n$, and $s_0>0$,
$$ \PP[\tau^{z_j}_{r_j}<\infty,1\le j\le n; \tau^{z_{1}}_{s_{0}}<\tau^{z_{2}}_{r_{2}}<\tau^{z_1}_{r_1}] \lesssim   F(z_1,\dots,z_n;r_1,\dots,r_n)\cdot\Big(\frac{s_{0}}{|z_{1}-z_{2}|\wedge |z_{1}|}\Big)^{\frac{\alpha}{32n^2}}. $$ 
Let $k_0\in \{2,\dots,n\}$ and $\beta=\frac{\alpha}{32n^2}$.  Applying (\ref{perm-Fzr}) to the above formula with a permutation $\sigma$ of $\N_n$, which sends $1$ to $k_0$ and $2$ to $1$, we find that
$$\PP[\tau^{z_j}_{r_j}<\infty,1\le j\le n; \tau^{z_{k_0}}_{s_{0}}<\tau^{z_{1}}_{r_{2}}<\tau^{z_{k_0}}_{s_0}] \lesssim   F(z_1,\dots,z_n;r_1,\dots,r_n)\cdot\Big(\frac{s_{0}}{|z_{k_0}-z_{1}|\wedge |z_{k_0}|}\Big)^{\beta}.$$
We then complete the proof by setting $z_1,\dots,z_n\in\R\sem\{0\}$.
\end{proof}

\begin{Proposition}
	Let $z_1  \in\R\sem \{0\}$  and $0\le s< r<R\wedge |z_1|$. On the event $\{\tau^{z_1}_ r<\tau^*_{z_1}\}$, let $ \xi_+ $ be the connected component  of $\{|z-z_1|=R\}\cap H_{\tau^{z_1}_ r}$ with one endpoint being $z_1+\sign(z_1)R$; otherwise let $\xi_+=\emptyset$. Let $$E_{r,s;R}=\{\gamma[\tau^{z_1}_ r,\tau^{z_1}_s]\cap \xi_+ =\emptyset \}.$$
Then
	\begin{enumerate}
		\item [(i)]  If $s>0$, $\PP_{z_1}^s[ E_{r,s;R}^c]\lesssim ({ r}/{R})^{\alpha}$.
		\item [(ii)] If $s=0$,  $\PP_{z_1}^*[E_{r,0;R}^c] \lesssim ({ r}/{R} )^{\alpha}$.
	\end{enumerate}\label{stayin}
\end{Proposition}
\begin{proof}
(i)  Assume that $z_1 >0$ by the left-right symmetry of chordal SLE. Suppose $\gamma$ follows the law $\PP$. Since $\kappa\in(0,8)$, the probability that $\gamma$ visits $\{z_1+ s,z_1-s,z_1+R\}$ is zero. We now assume that $\gamma$ does not visit this set.
 Let $\tau=\inf(\{t\ge \tau^{z_1}_r:\gamma(t)\in \xi_+\}\cup\{\infty\})$. Then $\tau$ is a stopping time, and $E_{r,s;R}=\{\tau<\tau^{z_1}_s<\infty\}$. Let $E_\tau=\{\tau<\tau^{z_1}_{s}\}\in \F_\tau$. By DMP of chordal SLE, conditionally on $\F_{\tau}$ and the event $E_\tau$, there is a random curve $\til \gamma$ following the law $\PP$ such that $\gamma(\tau+\cdot)=Z_\tau^{-1}\circ \til \gamma$.  Let $D=\{z\in\HH\sem K_\tau:|z-z_1|\le s\}$, $\til D=Z_\tau(D)$, $\til z_1=Z_\tau(z_1)>0$, and $\til s=\rad_{\til z_1}(\til D)>0$. On the event $E_\tau$, in order for $E_{r,s;R}$ to happen, we need that $\gamma(\tau+\cdot)$ visits $D$, which is equivalent to that $\til\gamma$ visits $\til D$. By (\ref{1pt}),
 $$\PP[E_{r,s;R}^c|\F_\tau,E_\tau]\lesssim (1\wedge (\til r_1/\til z_1))^\alpha.$$
 By Lemma \ref{lem-extremal2} and conformal invariance of extremal length,
 $$1\wedge ({\til r_1}/{\til z_1})\lesssim e^{-\pi d_{\HH}((-\infty, 0],\til D)}=e^{-\pi d_{H_\tau}(S^-_\tau,D)}.$$
 Since $S^-_\tau$ can be separated from $D$ in $H_\tau$ by the semi-annulus $\{s<|z-z_1|<R\}$, by the comparison principle of extremal length,
 $ d_{H_\tau}(S^-_\tau,D)\ge \log(R/s)/\pi$. So by the above two displayed formulas we get $\PP[E_{r,s;R}^c|\F_\tau,E_\tau]\lesssim (s/R)^\alpha$, which together with $\PP[E_\tau]\le \PP[\tau^{z_1}_ r<\infty]\lesssim ( r/|z_1|)^\alpha $ (the upper bound in (\ref{1pt}))   implies that $\PP[E_{r,s;R}^c]\lesssim   (   r /{|z_1|} )^\alpha  ( {s}/{R} )^{\alpha}$.
Combining this estimate with the lower bound in (\ref{1pt}), i.e., $\PP[\tau^{z_1}_{s}]\gtrsim (s/z_1)^\alpha$, we get (i).

	(ii) From Proposition \ref{RN<1} and (i), we get $\PP_{z_1}^*[E_{r,s;R}^c] \lesssim ( { r}/{R} )^{\alpha}$ for any $s\in (0, r)$. We then complete the proof by sending $s$ to $0^+$. 
\end{proof}


\begin{Lemma}
  Let $z_1,\dots,z_n,w_1,\dots,w_m$ be distinct points in $\R\sem\{0\}$, where $n\ge 1$ and $m\ge 0$. Suppose that all $z_j$ have the same sign $\sigma_z\in\{+,-\}$, all $w_k$ have the same sign $\sigma_w\in\{+,-\}$,  $\sigma_z\ne \sigma_w$, and both $j\mapsto |z_j|$ and $k\mapsto |w_k|$ are  increasing. Let $z_0=w_0=0$, $z_{n+1}=\sigma_z\cdot \infty$, and $w_{m+1}=\sigma_w\cdot \infty$.  Let $r_j\in(0,(|z_j-z_{j-1}|\wedge |z_j-z_{j+1}|)/2)$, $1\le j\le n$, and $s_k\in (0,(|w_k-w_{k+1}|\wedge |w_k-w_{k-1}|)/2)$, $1\le k\le m$. Let $R>2(|z_n|\vee |w_m|)$. Then
\begin{align}
&\PP[\tau^\infty_R<\tau^{z_j}_{r_j}<\infty,1\le j\le n;\tau^\infty_R<\tau^{w_k}_{s_k}<\infty,1\le k\le m] \nonumber\\
  \lesssim & \Big(\frac{|z_1|}{R}\Big)^\alpha F(z_1,\dots,z_n,w_1,\dots,w_m) \prod_{j=1}^n r_j^\alpha\cdot \prod_{k=1}^m s_k^\alpha.\label{PR}
  \end{align}
  \label{Lemma-PR}
\end{Lemma}
\begin{proof}
By symmetry, we may assume that $w_m<\cdots<w_1<0<z_1<\cdots <z_n$. Define $F_z$ and $F_w$ such that
$F_z= \prod_{j=1}^n |z_j-z_{j-1}|^{-\alpha}$;
$F_w =\prod_{k=1}^m |w_k-w_{k-1}|^{-\alpha}$, if $m\ge 1$; and $F_w=1$ if $m=0$.
Then we have $F(z_1,\dots,z_n,w_1,\dots,w_m)=F_z F_w$.

Let $\tau=\tau^\infty_R$. Let $E$ denote the event in (\ref{PR}). Then $E=E^\tau_*\cap E_\#$, where
\begin{align*}
  E^\tau_*&:=\{\tau^\infty_R<\tau^{z_j}_{r_j}\wedge \tau^*_{z_j},1\le j\le n;\tau^\infty_R<\tau^{w_k}_{s_k}\wedge \tau^*_{w_k},1\le k\le m\}\in\F_\tau;\\
  E_\#:&=\{ \tau^{z_j}_{r_j}<\infty,1\le j\le n; \tau^{w_k}_{s_k}<\infty,1\le k\le m \}.
\end{align*}

 Suppose the event $E_*^\tau$ occurs. Let $\til z_j=Z_\tau(z_j)$, $D_j=\{z\in H_\tau:|z-z_j|\le r_j\}$, $\til D_j=Z_\tau(D_j)$, and $\til r_j=\rad_{\til z_j}(\til D_j)$, $1\le j\le n$. Let $\til w_k=Z_\tau(w_k)$, $E_k=\{z\in H_\tau:|z-w_k|\le s_k\}$,   $\til E_k=Z_\tau(E_k)$, and $\til s_k=\rad_{\til w_k}(\til E_k)$, $1\le k\le m$. Then $\til w_m<\cdots <\til w_1<0<\til z_1<\cdots <\til z_n$.  By DMP of chordal SLE$_\kappa$ and Proposition \ref{lower-upper},
\begin{align}
   \PP[E_\#|\F_\tau,E_*^\tau] \lesssim &  \Big(1\wedge \frac{\til r_{n}}{|\til z_{n}|}\Big)^\alpha\cdot
\prod_{j=1}^{n-1} \Big(1\wedge \frac{\til r_j}{|\til z_j|\wedge |\til z_j-\til z_{j+1}|}\Big)^\alpha   \nonumber\\
\cdot& \Big(1\wedge \frac{\til s_m}{|\til w_m|}\Big)^\alpha\cdot \prod_{k=1}^{m-1} \Big(1\wedge  \frac{\til s_k}{|\til w_k|\wedge |\til w_k-\til w_{k+1}|}\Big)^\alpha. \label{EE0R}
\end{align}
Here we organize  $\til z_j$'s and $\til w_k$'s by $\til z_n,\dots,\til z_1,\til w_m,\dots,\til w_1$ when   applying Proposition \ref{lower-upper}.
In the case that $m=0$, the second line disappears.

By Proposition \ref{lem-extremal2} and conformal invariance of extremal distance,
\BGE 1\wedge \frac{\til r_{j}}{|\til z_{j}|}\lesssim e^{-\pi d_{\HH} ((-\infty,0],\til D_j)}=e^{-\pi d_{H_\tau}(S^-_\tau, D_j)},1\le j\le n;\label{disjRz}\EDE
\BGE 1\wedge \frac{\til r_{j}}{|\til z_{j}-\til z_{j+1}|}\lesssim e^{-\pi d_{\HH} ([\til z_{j+1},\infty),\til D_j)}=e^{-\pi d_{H_\tau}([z_{j+1},\infty), D_j)},1\le j\le n-1;\label{disj-1Rz}\EDE
\BGE 1\wedge \frac{\til s_{k}}{|\til w_{k}|}\lesssim e^{-\pi d_{\HH} ([0,+\infty),\til E_k)}=e^{-\pi d_{H_\tau}(S^+_\tau, E_k)},1\le k\le m;\label{disjRw}\EDE
\BGE 1\wedge \frac{\til s_{k}}{|\til w_{k}-\til w_{k+1}|}\lesssim e^{-\pi d_{\HH} ((-\infty,\til w_{k+1}],\til E_k)}=e^{-\pi d_{H_\tau}((-\infty,w_{k+1}], E_k)},1\le k\le m-1.\label{disj-1Rw}\EDE

\begin{figure}
\centering
\includegraphics[width=1\textwidth]{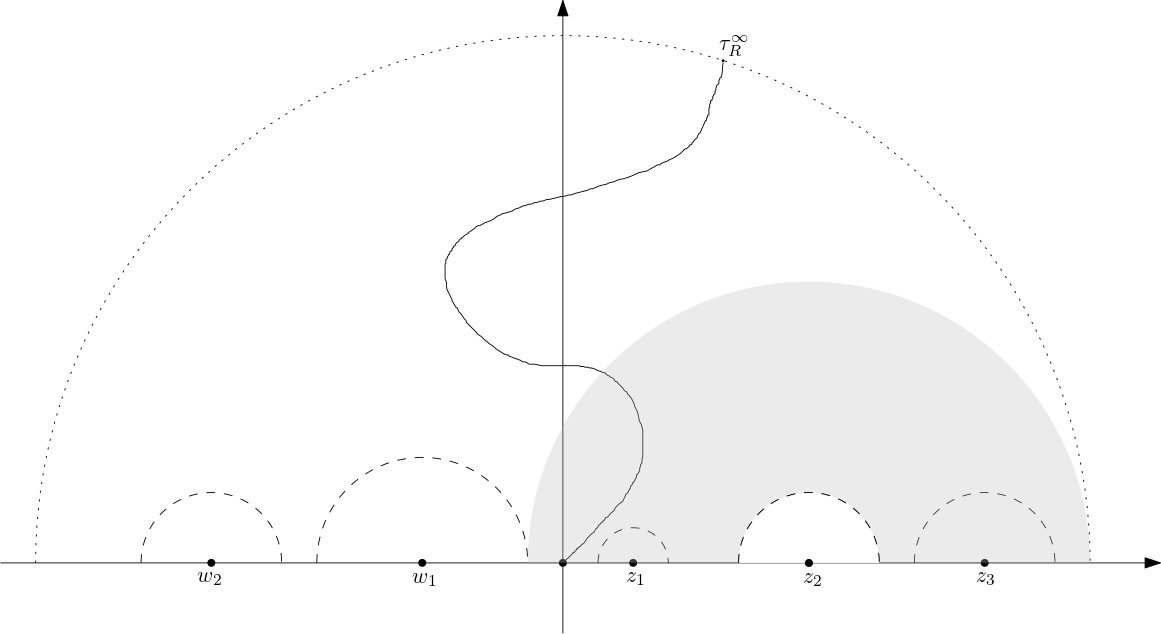}
\caption[A figure for the proof of Lemma \ref{Lemma-PR}]{{\textbf{A figure for the proof of Lemma \ref{Lemma-PR}.}} This figure illustrates an application of the comparison principle of extremal distance in the proof of Lemma \ref{Lemma-PR}. Here $n=3$ and $m=2$. The curve $\gamma$ is stopped at the time $\tau=\tau^\infty_R$. Assume that the event $E_*^\tau$ occurs. To bound the extremal distance $d_{H_\tau}(D_2,S^-_\tau)$ from below for example, we use the semi-annulus (shaded region) $A_2:=\{z\in\HH: r_2<|z-z_2|<R-|z_2|\}$ and the fact that any curve in $H_\tau$ that connects the semi-circle $\pa D_2\cap \HH$ with the left side of $\gamma[0,\tau]$ or the real interval $(-\infty,0]$ must cross $A_2$, i.e., contain a subpath in $A_2$ connecting its two semi-circles. The intersection of  $A_2$ with $\gamma[0,\tau]$  does not cause a problem in the application. }
\label{zwR}
\end{figure}

Since $S^-_\tau$ can be separated from $D_j$ in $H_\tau$ by $\{z\in\HH: r_j<|z-z_j|<R-|z_j|\}$, by comparison principle of extremal distance,
$$d_{H_\tau}(S^-_\tau, D_j)\ge \frac 1\pi \log\Big(\frac{R-|z_j|}{r_j}\Big),$$
which combined with (\ref{disjRz}) and that $R>2|z_j|$ implies that
\BGE 1\wedge \frac{\til r_{j}}{|\til z_{j}|}\lesssim \frac{r_j}{R-|z_j|}\asymp \frac{r_j}{R},\quad 1\le j\le n.\label{rjzjR}\EDE
See Figure \ref{zwR}.
Since $[z_{j+1},\infty)$ can be separated from $D_j$ in $H_\tau$ by   $\{z\in\HH: r_j<|z-z_j|<|z_{j+1}-z_j|\}$, by comparison principle of extremal distance,
$$d_{H_\tau}([z_{j+1},\infty), D_j)\ge \frac 1\pi \log\Big(\frac{|z_{j+1}-z_j|}{r_j}\Big),$$
which combined with (\ref{disj-1Rz}) implies that
\BGE 1\wedge \frac{\til r_{j}}{|\til z_{j}-\til z_{j+1}|}\lesssim \frac{r_j}{|z_{j+1}-z_j|},\quad 1\le j\le n-1.\label{rjzj+1R}\EDE
For $1\le j\le n-1$, since $R-|z_j|\ge |z_{j+1}|-|z_j|=|z_{j+1}-z_j|$, by (\ref{rjzjR}) and (\ref{rjzj+1R}),
\BGE 1\wedge \frac{\til r_j}{|\til z_j|\wedge |\til z_j-\til z_{j+1}|} \lesssim \frac{r_j}{|z_{j+1}-z_j|},\quad 1\le j\le n-1.\label{rjzjR-com}\EDE
Similarly,
\BGE 1\wedge \frac{\til s_{k}}{|\til w_{k}|}\lesssim \frac{s_k}{R-|w_k|}\asymp \frac{s_k}{R},\quad 1\le k\le m;\label{skwkR}\EDE
\BGE 1\wedge \frac{\til s_k}{|\til w_k|\wedge |\til w_k-\til w_{k+1}|} \lesssim \frac{s_k}{|w_{k+1}-w_k|},\quad 1\le k\le m-1.\label{skwkR-com}\EDE

Combining (\ref{EE0R}) with (\ref{rjzjR}) (for $j=n$), (\ref{rjzjR-com}), (\ref{skwkR}) (for $k=m$) and (\ref{skwkR-com}), we get
 \begin{align*}\PP[E_\#|\F_\tau,E_*^\tau]& \lesssim  \Big(\frac{r_n}{R }\Big)^\alpha \prod_{j=1}^{n-1} \Big(\frac{r_j}{|z_j-z_{j+1}|}\Big)^\alpha \cdot \Big(\frac{s_m}{R }\Big)^\alpha \prod_{k=1}^{m-1} \Big(\frac{s_k}{|w_k-w_{k+1}|}\Big)^\alpha\\
 &\le \Big(\frac{|z_1|}{R}\Big)^\alpha F_z F_w\prod_{j=1}^n r_j^\alpha \prod_{k=1}^m s_k^\alpha 
\end{align*}
Here, if $m=0$, the factors involving $s_k$ and $s_m$ disappear; if $m\ge 1$,  we used that $R\ge |w_1|$ in the estimate.  Since $ F(z_1,\dots,z_n,w_1,\dots,w_m)=F_zF_w$, taking expectation we get (\ref{PR}).
\end{proof}

\begin{Lemma}
  Suppose $x_0,\dots,x_N$, $N\ge 1$, are distinct points in $\R\sem\{0\}$ that have the same sign $\nu\in\{+,-\}$, and $j\mapsto |x_j|$ is increasing. Let  $x_{N+1}=\nu\cdot \infty$. Let $R_j=(|x_j-x_{j+1}|\wedge |x_j-x_{j-1}|)/2$ and $r_j\in (0,R_j)$, $1\le j\le N$. Let $r_0\in (0,|x_0-x_1|/2)$. Then
   \begin{align} &\PP[\tau^{x_j}_{r_j}<\tau^{x_0}_{r_0}<\infty;1\le j\le N]
   \lesssim     \Big( \frac{r_N}{|x_N|}\Big)^\alpha \prod_{k=0}^{N-1} \Big(  \frac{ r_k  }{|x_k-x_{k+1}|}\Big)^\alpha\cdot \prod_{k=1}^N  \Big(  \frac{ r_{k}  }{|x_k-x_{k-1}|}\Big)^\alpha \label{n-1}\\
\lesssim &  \Big( \frac{R_N}{|x_N|}\Big)^\alpha \Big(\frac{r_0}{|x_0-x_1|}\Big)^\alpha \prod_{k=1}^{N} \Big(  \frac{ r_k  }{R_k}\Big)^{2\alpha} \le   \Big(\frac{r_0}{|x_0-x_1|}\Big)^\alpha \prod_{k=1}^{N} \Big(  \frac{ r_k  }{R_k}\Big)^{2\alpha}. \label{n-1'}
   \end{align}
  \label{inner-last}
\end{Lemma}
\begin{proof}
Assume all $x_j$'s are positive by symmetry.
 Let $P$ denote the RHS of (\ref{n-1})  (depending on $x_0,\dots,x_N$ and $r_0,\dots,r_N$). We write $\tau_j$ for $\tau^{x_j}_{r_j}$, $1\le j\le N$. Let $S_N^*$ denote the set of permutation $\sigma$ of $ \{0,1,\dots,N\}$ such that $\sigma(n)=0$. For each $\sigma\in S_N^*$, let $E_\sigma=\{\tau_{\sigma(0)}<\tau_{\sigma(2)}<\cdots<\tau_{\sigma(N)}<\infty\}$.
 Then $\bigcup_{\sigma\in S_N^*} E_\sigma$ is the event in (\ref{n-1}). To prove (\ref{n-1}), it suffices to show that, for any $\sigma\in S_N^*$, $\PP[E_\sigma]\lesssim P$.

 Fix $\sigma\in S_N^*$. For $0\le k\le N-1$, let
 $$E^\sigma_k=\{\tau_{\sigma(0)}<\tau_{\sigma(1)}<\cdots <\tau_{\sigma(k)}<\tau_{\sigma(k+1)}\wedge \tau^*_{x_{\sigma(k+1)}}\}\in  \F_{\tau_{\sigma(k)}};$$
 and let $E^\sigma_N=E_\sigma$. Then
 $E^\sigma_0\supset E^\sigma_1\supset\cdots\supset E^\sigma_N=E_\sigma$.
Let
\BGE S_\sigma=\{j:n-1\ge j\ge \sigma^{-1}(n),\sigma(j+1)<\sigma(j)\}.\label{Tsigma}\EDE
For each $j\in S_\sigma$, let
\BGE S^\sigma_j=\{k : \sigma(j+1)<k<\sigma(j), \sigma^{-1}(k)<j\}.\label{Ssigma}\EDE
In plain words, $S_\sigma$ is the set of index $j\ge j_0$, where $j_0:= \sigma^{-1}(N)$, such that $\sigma(j+1)<\sigma(j)$; and $S^\sigma_j$ is the set of index $k$, which lies strictly between $\sigma(j+1)$ and $\sigma(j)$, such that the disc $\{|z-x_k|\le r_k\}$ was visited by $\gamma$ before $\{|z-x_{\sigma(j)}|\le r_{\sigma(j)}\}$. For example, $j_0$ and $N-1$ belong to $S^\sigma$. For $j\in S_\sigma$, the set $S^\sigma_j$ may or may not be empty. See Figure \ref{Figure-x}.

\begin{figure}
\centering
\includegraphics[width=1\textwidth]{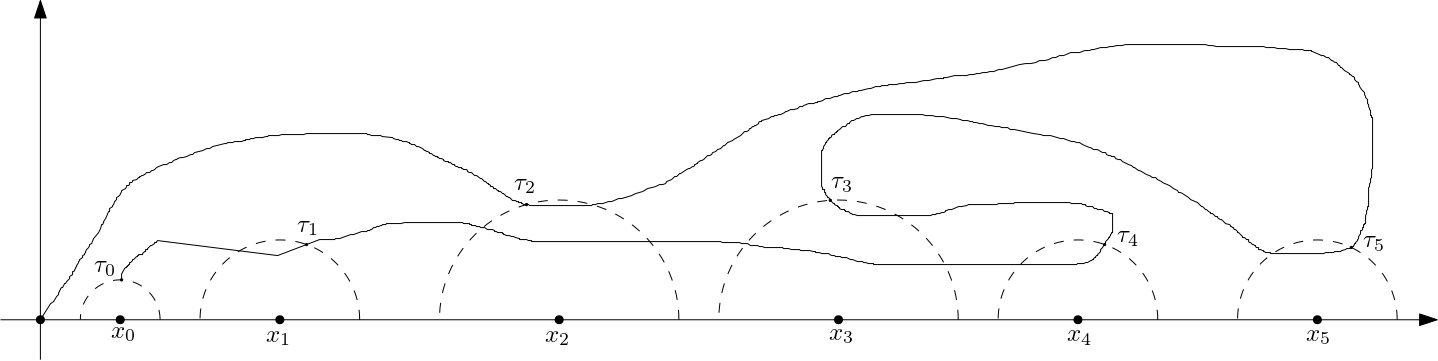}
\caption[The first figure for the proof of Lemma \ref{inner-last}]{\textbf{The first figure for the proof of Lemma \ref{inner-last}.} This figure illustrates a situation in the proof of Lemma \ref{inner-last}. Here $n=5$, and the event $E_\sigma$ happens, where $\sigma = \bigl(\begin{smallmatrix}
0 &  1 & 2 & 3 & 4 & 5 \\
  2 & 5 & 3 & 4 &  1  & 0\end{smallmatrix}\bigr)$. We have $S_\sigma=\{1,3,4\}$ since $ \sigma^{-1}(5)=1$, $\sigma(1)=5>3=\sigma(2)$, $\sigma(3)=4>1=\sigma(4)$, $\sigma(4)=1>0=\sigma(5)$, but $\sigma(2)=3<4=\sigma(3)$. We have $S^\sigma_1=\emptyset$ because the only index between $\sigma(2)$ and $\sigma(1)$ is $4$, and $\sigma^{-1}(4)=3>2$. We have $S^\sigma_3=\{2,3\}$ because $2,3$ lie between $\sigma(4)$ and $\sigma(3)$, and $\sigma^{-1}(2),\sigma^{-2}(3)<3$. We have $S^\sigma_4=\emptyset$ because there is no index that lies between $\sigma(5)$ and $\sigma(4)$.   }
\label{Figure-x}
\end{figure}

For $0\le j\le N-1$, let $Q^+_j=(  \frac{2r_j}{|x_j-x_{j+1}|})^\alpha$. For $1\le j\le N$, let $Q^-_j=( \frac{2r_j}{|x_j-x_{j-1}|})^\alpha$. Let $Q_n=( \frac{r_N}{|x_N|})^\alpha$. Then
$P= Q_N\cdot \prod_{j=1}^{N-1} Q^+_j\cdot \prod_{j=2}^N Q^-_j$. By (\ref{1pt}),
\BGE \PP[E^\sigma_{\sigma^{-1}(N)}]\le \PP[\tau_N<\infty]\lesssim Q_N.\label{PQn}\EDE
We claim that, for any  $j\in S_\sigma$,
\BGE \PP[E^\sigma_{j+1}|\F_{\tau_{\sigma(j)}},E^{\sigma}_j]\lesssim Q_{\sigma(j)}^- Q_{\sigma(j+1)}^+ \prod_{k\in S^\sigma_j}(Q_{k}^+Q_{k}^-);\label{PQj}\EDE
and
\BGE \prod_{j\in S_\sigma}\Big(Q_{\sigma(j)}^- Q_{\sigma(j+1)}^+ \prod_{k\in S^\sigma_j}(Q_{k}^+Q_{k}^-)\Big)\le \prod_{l=0}^{N-1} Q^+_l\cdot \prod_{l=1}^N Q^-_l.\label{PQj-union}\EDE
Note that (\ref{PQn},\ref{PQj},\ref{PQj-union}) together   imply that $\PP[E_\sigma]=\PP[E^\sigma_n]\lesssim P$.

We first prove (\ref{PQj-union}). It suffices to show that
\BGE \{0,\dots,N-1\}\subset \bigcup_{j\in S_\sigma} (  \{\sigma(j+1)\}\cup S^\sigma_j);\label{Q1n-1}\EDE
\BGE \{1,\dots,N\}\subset \bigcup_{j\in S_\sigma} (  \{\sigma(j)\}\cup S^\sigma_j).\label{Q2n}\EDE
Let $l\in\{0,\dots,N-1\}$. We consider several cases. Case 1. $\sigma^{-1}(l)< \sigma^{-1}(N)$.  Since $\sigma^{-1}(0)=N$, we have $l\ge 1$. Since $\sigma(\sigma^{-1}(N))=N>l$ and $\sigma(N)=0<l$, there exists $\sigma^{-1}(N)\le j_0\le N-1$ such that $\sigma(j_0)>l>\sigma(j_0+1)$. By (\ref{Tsigma},\ref{Ssigma}) we have $j_0\in S_\sigma$ and $l\in S^\sigma_{j_0}$. Case 2. $\sigma^{-1}(l)\ge \sigma^{-1}(N)$. Then $\sigma^{-1}(l)-1\ge \sigma^{-1}(N)$ since $l\ne N$. Consider two subcases. Case 2.1. $\sigma(\sigma^{-1}(l)-1)>\sigma(\sigma^{-1}(l))=l$. In this subcase, $j_1:=\sigma^{-1}(l)-1\in S_\sigma$ by (\ref{Tsigma}), and $\sigma(j_1+1)=l$. Case 2.2. $\sigma(\sigma^{-1}(l)-1)<\sigma(\sigma^{-1}(l))=l$. Since $\sigma(\sigma^{-1}(N))=N>l>\sigma(\sigma^{-1}(l)-1)$ and $\sigma^{-1}(N)\le \sigma^{-1}(l)-1$, there exists $\sigma^{-1}(N)\le j_2\le \sigma^{-1}(l)-2$ such that $\sigma(j_2)>l>\sigma(j_2+1)$. This implies that $j_2\in S_\sigma$ and $l\in S^\sigma_{j_2}$. Thus, in all cases, there is some $j\in S_\sigma$ such that $l\in \{\sigma(j+1)\}\cup S^\sigma_j$. So we get (\ref{Q1n-1}).

Let $l\in\{1,\dots,N\}$. We consider several cases. Case 1. $\sigma^{-1}(l)< \sigma^{-1}(N)$. Then $l\le N-1$. By  Case 1 of the last paragraph, there exists $j_0\in S_\sigma$ such that $l\in S^\sigma_{j_0}$. Case 2. $\sigma^{-1}(l)\ge \sigma^{-1}(N)$. Consider two subcases. Case 2.1. $\sigma(\sigma^{-1}(l))>\sigma(\sigma^{-1}(l)+1)$. In this subcase, $j_1:=\sigma^{-1}(l)\in S_\sigma$ and $l=\sigma(j_1)$. Case 2.2. $l=\sigma(\sigma^{-1}(l))<\sigma(\sigma^{-1}(l)+1)$. Since $\sigma(\sigma^{-1}(l)+1)>l>0=\sigma(N)$, there exists $\sigma^{-1}(l)+1\le j_2\le N-1$ such that $\sigma(j_2)>l>\sigma(j_2+1)$. This implies that $j_2\in S_\sigma$ and $l\in S^\sigma_{j_2}$. Thus, in all cases, there is some $j\in S_\sigma$ such that $l\in \{\sigma(j)\}\cup S^\sigma_j$. So we get (\ref{Q2n}). Combining (\ref{Q1n-1},\ref{Q2n}) we get (\ref{PQj-union}).

\begin{figure}
\centering
\includegraphics[width=1\textwidth]{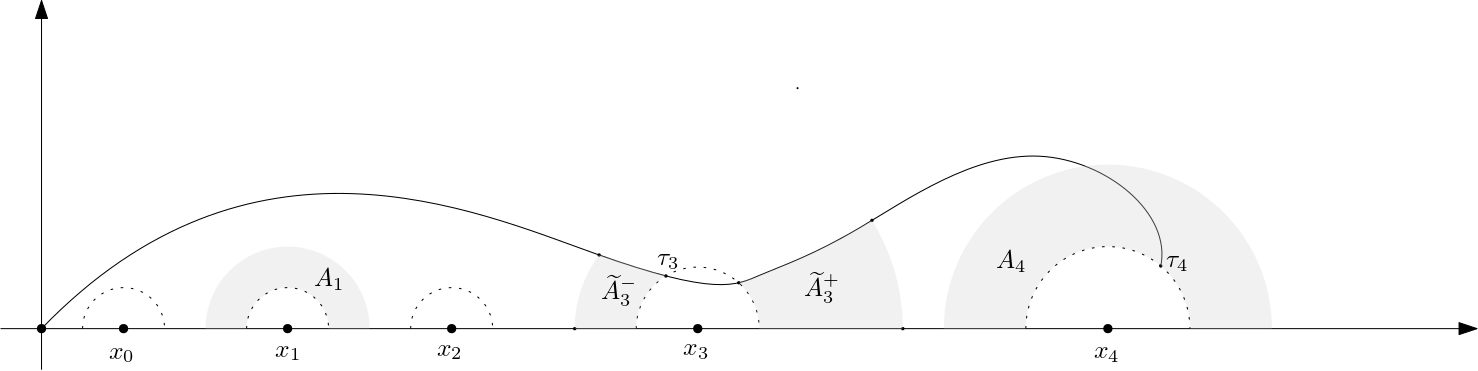}
\caption[The second figure for the proof of Lemma \ref{inner-last}]{\textbf{The second figure for the proof of Lemma \ref{inner-last}.} This figure illustrates an application of the comparison principle of extremal distance in the proof of Lemma \ref{inner-last}. Here $n=4$, and the event $E_\sigma$ happens, where $\sigma = \bigl(\begin{smallmatrix}
0 &  1 & 2 & 3 & 4  \\
  3& 4 & 1 & 2 &  0 \end{smallmatrix}\bigr)$.  We stop the curve at the time $\tau:=\tau_4$. Then the next semi-disc to visit is $D_1=\{z\in\HH:|z-x_1|\le r_1\}$. We know that $1\in S_\sigma$, $\sigma(1)=4$, $\sigma(2)=1$, and $S^\sigma_1=\{3\}$. The  $D_1$ is separated from $S^-_{\tau}$ in $H_\tau$ by the disjoint regions $A_1$, $A_4$, $\til A_3^+$ and $\til A_3^-$, among which $A_1$ and $A_4$ are semi-annuli, and $\til A_3^+$ and $\til A_3^-$ are subsets of two semi-annuli, which have the same center $x_3$, same inner radius $r_3$, but different outer radii.   }
\label{x2}
\end{figure}

Finally, we prove (\ref{PQj}). Fix $j\in S_\sigma$. Let $\tau=\tau_{\sigma(j)}$. Suppose the event $E^{\sigma}_j$ occurs. Let $w=x_{\sigma(j+1)}$, $D=\{z\in\HH: |z-w|\le r_{\sigma(j+1)}\}$, $\til w=Z_\tau(w)$, $\til D=Z_\tau(\til D)$, and $\til r=\rad_{\til w}(\til D)$. By DMP of chordal SLE$_\kappa$ and (\ref{1pt}),
\BGE \PP[E^\sigma_{j+1}|\F_\tau,E^\sigma_j]\le \PP[\tau_{\sigma(j+1)}<\infty |\F_\tau,E^\sigma_j]\lesssim \Big(1\wedge \frac{\til r}{|\til w|}\Big)^\alpha.\label{1wedge>}\EDE
By Proposition \ref{lem-extremal2} and conformal invariance of extremal distance,
\BGE 1\wedge \frac{\til r}{|\til w|}\lesssim e^{-\pi d_{\HH}((-\infty,0],\til D)} =e^{-\pi d_{H_\tau}(S^-_\tau, D)}.\label{1wedge<}\EDE
Define  semi-annuli
\begin{align*}
  A_{\sigma(j)}&=\{z\in\HH: r_{\sigma(j)} <|z-x_{\sigma(j)}|<|x_{\sigma(j)-1}-x_{\sigma(j)}|/2\};\\
A_{\sigma(j+1)}&=\{z\in\HH: r_{\sigma(j+1)}  <|z-x_{\sigma(j+1)}|<|x_{\sigma(j+1)+1}-x_{\sigma(j+1)}|/2\};
\\
A_k^\pm&=\{z\in\HH : r_k  <|z-x_k|<|x_{k\pm 1}-x_k|/2\},\quad k\in S^\sigma_j.
\end{align*}
For each $k\in S^\sigma_j$, define $\til A_k^\pm$ to be the connected component of $A_k^\pm \cap H_\tau$ whose boundary contains $x_k\pm r_k$. Then $A_{\sigma(j)}$, $A_{\sigma(j+1)}$, $\til A_k^+$ and $\til A_k^-$, $k\in S^\sigma_j$, are mutually disjoint. See Figure \ref{x2}. Since the event $E^\sigma_j$ occurs, any curve in $H_\tau$ connecting $D$ with $S^-_\tau$ must contain a subarc crossing $A_{\sigma(j)}$, a subarc crossing $A_{\sigma(j+1)}$, a subarc contained in $\til A_k^+$ crossing $A_k^+$ for each $k\in S^\sigma_j$, and a subarc contained in $\til A_k^-$ crossing $A_k^-$ for each $k\in S^\sigma_j$. By the comparison principle and composition rule of extremal length, we know that
\begin{align}d_{H_\tau}(S^-_\tau, D)\ge & \frac 1\pi\log\Big(\frac{|x_{\sigma(j)}-x_{\sigma(j)-1}|}{2r_{\sigma(j)} } \Big) +\frac 1\pi\log\Big(\frac{|x_{\sigma(j+1)}-x_{\sigma(j+1)+1}|}{2r_{\sigma(j+1)}} \Big) \nonumber\\ &+
\frac 1\pi \sum_{k\in S^\sigma_j} \log\Big(\frac{|x_k-x_{k+1}|}{2r_k }\Big) +
\frac 1\pi \sum_{k\in S^\sigma_j} \log\Big(\frac{|x_k-x_{k-1}|}{2r_k}\Big).\label{ext-lower}
\end{align}
Combining (\ref{1wedge>},\ref{1wedge<},\ref{ext-lower}) we get (\ref{PQj}). Then we get (\ref{n-1}), which implies (\ref{n-1'}) because $|x_k-x_{k\pm 1}|\ge R_k$, $1\le k\le N$, and $R_N=|x_N-x_{N-1}|\le |x_N|$.
\end{proof}

\begin{Remark} By a slight modification of the above proof, we can obtain the following estimate. Let $x_0,\dots,x_{N+1},R_1,\dots,R_N,r_0,\dots,r_N$ be as in Lemma \ref{inner-last}. Let $I=[a,x_0]$ for some $a\in (0,x_0)$, and $\tau^{I}_{r_0}=\tau_{I\times [0,r_0]}$. Then
\begin{align}&\PP[\tau^{z_j}_{r_j}<\tau^{I}_{r_0}<\infty;1\le j\le N]
   \lesssim   \prod_{j=1}^N \Big( \frac{r_j}{R_j}\Big)^{2\alpha}.\nonumber
   \end{align}
To prove the estimate, we may use the same extremal length argument except that we do not use a semi-annulus centered at $x_0$ because such a semi-annulus may not disconnect $I\times [0,r_0]$ from other $x_j$'s in $H_\tau$. So we have the same factor in the upper bound except for $(\frac{r_0}{|x_1-x_0|})^\alpha$. 
  \label{after-inner-last}
\end{Remark}

\begin{Lemma}
Let $z_j$, $0\le j\le n+1$, $w_k$, $0\le k\le m+1$, $r_j$, $1\le j\le n$, $s_k$, $1\le k\le m$, be as in Lemma \ref{Lemma-PR}. Now assume $n\ge 2$.
Let $j_0\in\{2,\dots,n\}$. 
Let
\begin{align*}
  Q= & |z_{j_0-1}-z_{j_0}|^{-\alpha} \cdot \prod_{j=1}^{j_0-1} |z_j-z_{j+1}|^{-\alpha} \cdot |w_m|^{-\alpha}\cdot \prod_{k=1}^{m-1}(|w_k|\wedge |w_k-w_{k+1}|)^{-\alpha}\\
  &\cdot (|z_{j_0}|\wedge |z_{j_0}-z_{j_0+1}|)^{-\alpha}\cdot \prod_{j=j_0+1}^{n} (|z_j-z_{j-1}|\wedge |z_j -z_{j+1}|)^{-\alpha}.
\end{align*}
Here when $m=0$, the $|w_m|^{-\alpha}\cdot \prod_{k=1}^{m-1}(|w_k|\wedge |w_k-w_{k+1}|)^{-\alpha}$ disappears; and when $j_0=n$, the  $\prod_{j=j_0+1}^{n} (|z_j-z_{j-1}|\wedge |z_j -z_{j+1}|)^{-\alpha}$ disappears. Then we have
    \begin{align}
    & \PP[\tau^{z_{j_0}}_{r_{j_0}}<\tau^{z_j}_{r_j}<\infty, j\in\N_n\sem\{j_0\}; \tau^{z_{j_0}}_{r_{j_0}}<\tau^{w_k}_{s_k} <\infty,k\in\N_m]
    \lesssim  Q r_{j_0}^\alpha \cdot \prod_{j=1}^n r_j^\alpha \cdot \prod_{k=1}^m s_k^\alpha.\label{ping-pong-ineq'}
  \end{align}
 \label{ping-pong}
\end{Lemma}
\begin{proof}
By symmetry, we may assume that $w_m<\cdots<w_1<0<z_1<\cdots <z_n$.  
Let $\tau=\tau^{z_{j_0}}_{r_{j_0}}$. Let $E$  denote the event in (\ref{ping-pong-ineq'}). See Figure \ref{zwj}. Let
\begin{align*}
  E_*^\tau&=\{\tau<\tau^{z_j}_{r_j}\wedge \tau^*_{z_j}: j\in\N_n\sem\{j_0\}\}\in\F_\tau;\\
    E_\#&=\{\tau^{z_j}_{r_j}<\infty,j\in\N_n\sem\{j_0\}; \tau^{w_k}_{s_k}<\infty,1\le k\le m\}.
\end{align*}
Then $E=E_*^\tau\cap E_\#$. By (\ref{1pt}), $\PP[E_*^\tau]\lesssim (r_{j_0}/|z_{j_0}|)^\alpha$.

\begin{figure}
\centering
\includegraphics[width=1\textwidth]{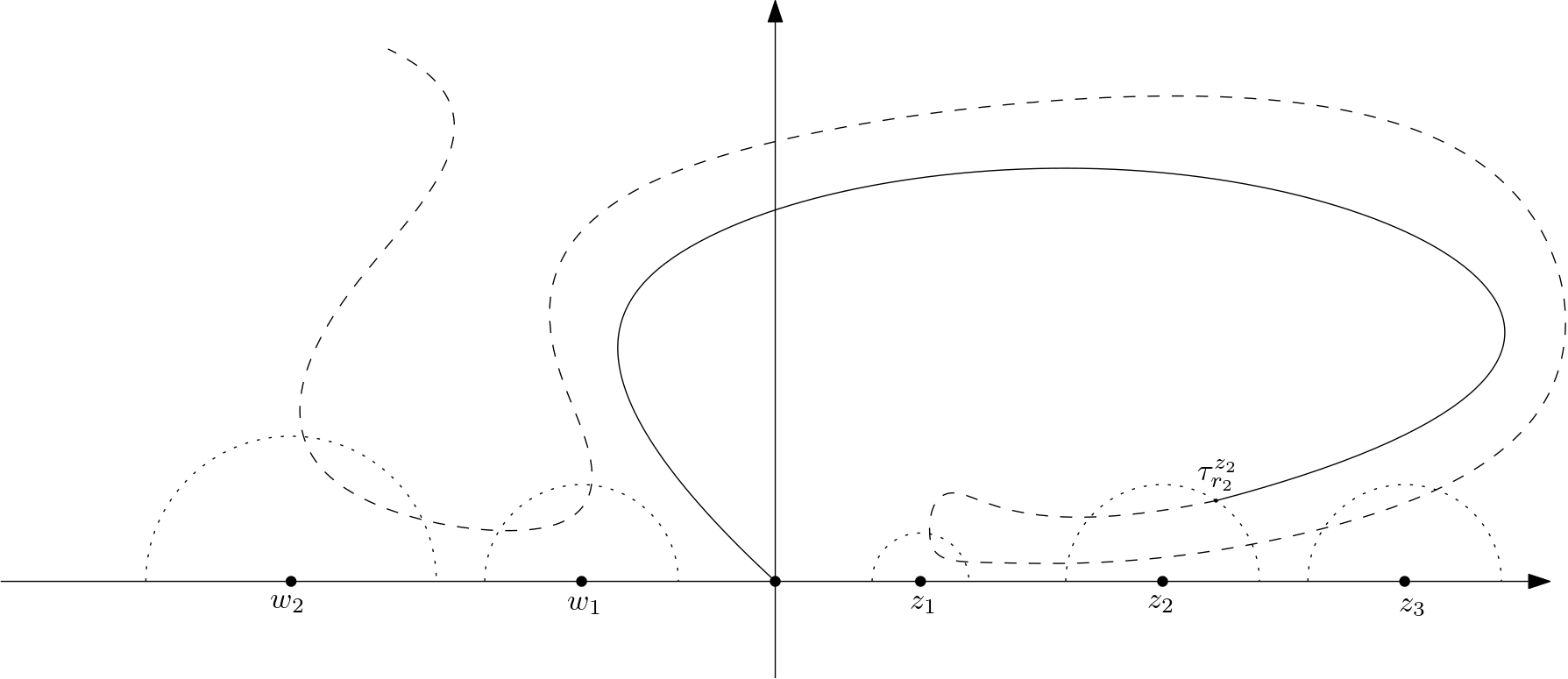}
\caption[A figure for the proof of Lemma \ref{ping-pong}]{\textbf{A figure for the proof of Lemma \ref{ping-pong}.} This figure illustrates the event $E$ in  Lemma \ref{ping-pong}. Here $n=3$, $m=2$, and $j_0=2$. The curve $\gamma$ visits the five semi-discs centered at $z_1,z_2,z_3,w_1,w_2$, among which the one centered at $z_2$ is first visited (at the time $\tau=\tau^{z_2}_{r_2}$). The parts of $\gamma$ before $\tau$ and after $\tau$ are respectively drawn in solid and dashed lines.  }
\label{zwj}
\end{figure}

 Suppose the event $E_*^\tau$ occurs. Let $\til z_j=Z_\tau(z_j)$, $D_j=\{z\in H_\tau:|z-z_j|\le r_j\}$, $\til D_j=Z_\tau(D_j)$, and $\til r_j=\rad_{\til z_n}(\til D_n)$, $1\le j\le n$. Let $\til w_k=Z_\tau(w_k)$, $E_k=\{z\in H_\tau:|z-w_k|\le s_k\}$,   $\til E_k=Z_\tau(E_k)$, and $\til s_k=\rad_{\til w_m}(\til E_m)$, $1\le k\le m$. Then $\til w_m<\cdots <\til w_1<0<\til z_1<\cdots <\til z_n$.  By DMP of chordal SLE$_\kappa$ and Proposition \ref{lower-upper} and that $E=E_*^\tau\cap E_\#$,
\begin{align}
   \PP[E|\F_\tau,E_*^\tau] \lesssim &  
\prod_{j=1}^{j_0-1} \Big(1\wedge \frac{\til r_j}{|\til z_j|\wedge |\til z_j-\til z_{j+1}|}\Big)^\alpha  \cdot \prod_{j=j_0+1}^n \Big(1\wedge \frac{\til r_j}{|\til z_j-\til z_{j-1}|}\Big)^\alpha\nonumber\\
\cdot& \Big(1\wedge \frac{\til s_m}{|\til w_m|}\Big)^\alpha\cdot \prod_{k=1}^{m-1} \Big(1\wedge  \frac{\til s_k}{|\til w_k|\wedge |\til w_k-\til w_{k+1}|}\Big)^\alpha. \label{EE0}
\end{align}
Here when applying Proposition \ref{lower-upper}, we ordered the points $\til z_j$  and $\til w_k$  by $$\til z_{j_0},\dots,\til z_1, \til z_{j_0+1},\dots,\til z_n,\til w_m,\dots,\til w_1,$$ and omit the factor $(1\wedge \frac{\til r_{j_0}}{|\til z_{j_0}|} )^\alpha$, which is bounded by $1$.

 By Proposition \ref{lem-extremal2} and conformal invariance of extremal length,
 \BGE 1\wedge \frac{\til r_j}{|\til z_j|} \lesssim e^{-\pi d_{\HH}((-\infty,0],\til D_j)}=e^{-\pi d_{H_\tau} (S^-_\tau,D_j)} ,\quad 1\le j\le j_0-1;\label{jD}\EDE
 \BGE 1\wedge \frac{\til s_k}{|\til w_k|}\lesssim e^{-\pi d_{\HH}([0,+\infty),\til E_k)}=e^{-\pi d_{H_\tau }(S^+_\tau,\E_k)},\quad 1\le k\le m;\label{kE}\EDE
 \BGE 1\wedge \frac{\til r_j}{|\til z_j-\til z_{j+1}|}\lesssim e^{-\pi d_{\HH}([\til z_{j+1},\infty),\til D_j)}=e^{-\pi d_{H_\tau}([z_{j+1},\infty),D_j)},\quad 1\le j\le j_0-1 ;\label{j+D}\EDE
 \BGE 1\wedge \frac{\til r_j}{|\til z_j-\til z_{j-1}|}\lesssim e^{-\pi d_{\HH}((-\infty,\til z_{j-1}],\til D_j)}= e^{-\pi d_{H_\tau}(K_\tau\cup (-\infty,z_{j-1}],D_j)},\quad j_0+1\le j\le n;\label{j-D}\EDE
 \BGE 1\wedge \frac{\til s_k}{|\til w_k-\til w_{k+1}|}\lesssim e^{-\pi d_{\HH}((-\infty,\til w_{k+1}],\til E_k)}=e^{-\pi d_{H_\tau}((-\infty, w_{k+1}],E_k)},\quad 1\le k\le m-1.\label{k+E}\EDE

Since $S^+_\tau$ and $E_m$ are separated by the semi-annulus $\{z\in\HH: s_m<|z-w_m|<|w_m|\}$ in $H_\tau$,  we have
$d_{H_\tau} (S^+_\tau,E_m)\ge \frac 1\pi \log(\frac{|w_m|}{s_m})$, which together  with (\ref{kE}) implies that
\BGE 1\wedge \frac{\til s_m}{|\til w_m|}\lesssim  \frac{s_m}{|w_m|}.\label{mE1}\EDE
For $1\le j\le j_0-2$, since $D_j$ is separated from both $S^-_\tau$ and $[z_{j+1},\infty)$ by the semi-annulus  $\{z\in\HH: r_j<|z-z_j|<|z_j-z_{j+1}|\}$ in $H_\tau$, we have
$$d_{H_\tau} (S^-_\tau,D_j),d_{H_\tau} ([z_{j+1},+\infty),D_j) \ge\frac 1\pi \log\Big(\frac{ |z_j-z_{j+1}|}{r_j}\Big),$$
which combined with (\ref{jD},\ref{j+D}) implies that
\BGE 1\wedge \frac{\til r_j}{|\til z_j|\wedge |\til z_j-\til z_{j+1}|}=\Big(1\wedge \frac{\til r_j}{|\til z_j| }\Big)\vee \Big(1\wedge \frac{\til r_j}{  |\til z_j-\til z_{j+1}|}\Big)\lesssim \frac{r_j}{|z_j-z_{j+1}|}. \label{jD1}\EDE
For $j=j_0-1$, we have a better estimate. Since $D_{j_0-1}$ is separated from both $S^-_\tau$ and $[z_{j_0},\infty)$ by a disjoint pair of semi-annuli $\{z\in\HH: r_{j_0-1}<|z-z_{j_0-1}|<|z_{j_0-1}-z_{j_0}|/2\}$ and $\{z\in\HH: r_{j_0}<|z-z_{j_0}|<|z_{j_0-1}-z_{j_0}|/2\}$, we have
$$d_{H_\tau} (S^-_\tau,D_{j_0-1}),d_{H_\tau} ([z_{j_0},+\infty),D_{j_0-1}) \ge \frac 1\pi  \log\Big( \frac{|z_{j_0-1}-z_{j_0}|}{2r_{j_0-1}}\Big)+\frac 1\pi \log \Big( \frac{|z_{j_0-1}-z_{j_0}|}{2r_{j_0}})\Big),$$
which combined with (\ref{jD},\ref{j+D}) implies that
\BGE 1\wedge \frac{\til r_{j_0-1}}{|\til z_{j_0-1}|\wedge |\til z_{j_0-1}-\til z_{j_0}|} \lesssim \frac{r_{j_0-1} }{|z_{j_0-1}-z_{j_0}| }\cdot  \frac{r_{j_0} }{|z_{j_0-1}-z_{j_0}| }. \label{jD1'}\EDE
For   $1\le k\le m-1$, since $E_k$ is separated from $S^+_\tau$ by   $\{z\in\HH:s_k<|z-w_k|<|w_k|\}$ in $H_\tau$, we get
$d_{H_\tau} (S^+_\tau,E_k) \ge \frac 1\pi \log (\frac{|w_k |}{2s_k})$.
Since $E_k$ is separated from $(-\infty,w_{k+1}]$ by  $\{z\in\HH:s_k<|z-w_k|<|w_k-w_{k+1}|\}$ in $H_\tau$, we get
$  d_{H_\tau} ((-\infty, w_{k+1}],E_k)  \ge \frac 1\pi \log (\frac{|w_k-w_{k+1}|}{2s_k}) $.
These two lower bounds of extremal lengths combined with (\ref{kE},\ref{k+E}) imply that
\BGE 1\wedge \frac{\til s_k}{|\til w_k|\wedge |\til w_k-\til w_{k+1}|}\lesssim \frac{s_k}{|w_k-w_{k+1}|\wedge |w_k|}.\label{kE1}\EDE

Suppose $j_0=n$.
Combining   (\ref{EE0},\ref{mE1}-\ref{kE1}), we get
$$\PP[E|\F_\tau,E_*^\tau]\lesssim  \Big(\frac{r_{j_0}}{|z_{j_0-1}-z_{j_0}|}\Big)^\alpha \Big(\frac{s_m}{|w_m |}\Big)^\alpha \cdot \prod_{j=1}^{j_0-1} \Big(\frac{r_j}{|z_j-z_{j+1}|}\Big)^\alpha \cdot \prod_{k=1}^{m-1} \Big(\frac{s_k}{|w_k-w_{k+1}|\wedge|w_k|}\Big)^\alpha,$$ which together with $\PP[E_*^\tau]\lesssim (r_{j_0}/|z_{j_0}  |)^\alpha$ and $z_{j_0+1}=\infty$ implies (\ref{ping-pong-ineq'}) for $j_0=n$.


Now suppose $2\le j_0\le n-1$. Let $\N_{(j_0,n]}=\{j_0+1,\dots,n\}$. For $j\in\N_{(j_0,n]}$, let $R_j=(|z_j-z_{j-1}|\wedge |z_j-z_{j+1}|)/2$. For each $\ulin k=(k_{j_0+1},\dots,k_n)\in (\N\cup\{0\})^{\N_{(j_0,n]}}$, let $S_{\ulin k}=\{j\in\N_{(j_0,n]}:K_j\ge 1\}$, and
$E_{\ulin k}$ denote the event that $\tau<\infty$ and $\dist(z_j,K_\tau)\ge R_j$, for $j\in \N_{(j_0,n]}\sem S_{\ulin k}$, and $R_j e^{-k_j}\le \dist(z_j,K_\tau)<R_j e^{1-k_j}$ for $j\in S_{\ulin k}$.

We now bound $\PP[E_{\ulin k}]$. If $S_{\ulin k}=\emptyset$, we use (\ref{1pt}) to conclude that
$$\PP[E_{\ulin k}]\le \PP[\tau^{z_{j_0}}_{r_0}<\infty]\lesssim (r_{j_0}/|z_{j_0}|)^\alpha.$$
Suppose $S_{\ulin k}\ne \emptyset$. We express $S_{\ulin k}=\{j_1<\cdots<j_N\}$. Let $x_s=z_{j_s}$, $0\le s\le N$. By the definition of $E_{\ulin k}$ and Lemma \ref{inner-last}, we have
\begin{align*}
  &\PP[E_{\ulin k}]\le \PP[\tau^{x_s}_{e^{1-k_{j_s}} R_{j_s}}<\tau^{x_0}_{r_{j_0}}<\infty,1\le s\le N]\\
  \lesssim &\Big(\frac{r_{j_0}}{|x_1-x_0|}\Big)^\alpha \prod_{s=1}^N \Big(\frac{e^{1-k_{j_s}} R_{j_s}}{  R_{j_s}}\Big)^{2\alpha}\lesssim  \Big( \frac{r_{j_0}}{|z_{j_0}-z_{j_0+1}|}\Big)^\alpha \prod_{j=j_0+1}^n  e^{-2\alpha k_j}.
\end{align*}
Combining the two formulas, we conclude that, for any $\ulin k \in (\N\cup\{0\})^{\N_{(j_0,n]}}$
 \BGE \PP[E_{\ulin k}]\lesssim 
 \Big( \frac{r_{j_0}}{|z_{j_0}|\wedge |z_{j_0}-z_{j_0+1}|}\Big)^\alpha \prod_{j=j_0+1}^n  e^{-2\alpha k_j}.\label{PEk}\EDE

Suppose for some $\ulin k=(k_{j_0+1},\dots,k_n)$, $E_{\ulin k}\cap E_*$ happens. We claim that
\BGE 1\wedge \frac{\til r_j}{|\til z_j-\til z_{j-1}|}\lesssim   \frac{r_j}{R_j e^{-k_j}},\quad j_0+1\le j\le n.\label{j-D=>}\EDE
Let $j\in\N_{(j_0,n]}$. First, (\ref{j-D=>}) holds trivially if $r_j\ge  R_j e^{-k_j}$.
Suppose that $r_j< R_j e^{-k_j}$. Then $D_j$ can be disconnected from $K_\tau$ and $(-\infty,z_{j-1}]$ in $H_\tau$ by  $\{z\in\HH:r_j<|z|<R_j e^{-k_j}\}$. By comparison principle of extremal distance, we have
$$d_{H_\tau}(K_\tau\cup (-\infty,z_{j-1}],D_j)\ge \frac 1\pi \log\Big(\frac{R_j e^{-k_j}}{r_j}\Big),$$
which together with (\ref{j-D}) implies (\ref{j-D=>}). So the claim is proved.

Combining   (\ref{EE0},\ref{mE1}-\ref{j-D=>}) and that $R_j=(|z_j-z_{j-1}|\wedge |z_j-z_{j+1}|)/2$, we get
\begin{align*} \PP[E\cap E_{\ulin k}]\lesssim&  \Big(\frac{r_{j_0} }{|z_{j_0}-z_{j_0-1}| }\Big)^\alpha \Big(\frac{s_m}{|w_m|}\Big)^\alpha \cdot \prod_{k=1}^{m-1} \Big(\frac{s_k}{|w_k-w_{k+1}|\wedge |w_k|}\Big)^\alpha\cdot \prod_{j=1}^{j_0-1}\Big( \frac{r_j}{|z_j-z_{j+1}|}\Big)^\alpha  \\ \cdot &  \Big( \frac{r_{j_0}}{|z_{j_0}|\wedge |z_{j_0}-z_{j_0+1}|}\Big)^\alpha 
\prod_{j=j_0+1}^n (|z_j-z_{j-1}|\wedge |z_j-z_{j+1}|)^{-\alpha} \cdot \prod_{j=j_0+1}^n e^{-\alpha k_j}. 
\end{align*}
Summing the inequality over $\ulin k\in(\N\cup\{0\})^{\N_{(j_0,n]}}$, we get (\ref{ping-pong-ineq'}).
%
\end{proof}

\begin{Lemma}
Let $n,m,j_0$, $z_0,\dots,\ha z_{j_0},\dots,z_{n+1}$, and $w_0,\dots,w_{m+1}$ be as in Lemma \ref{ping-pong}. Here the symbol $\ha z_{j_0}$ means that $z_{j_0}$ is missing in the list. Let $I$ be a compact real interval that lies strictly between $z_{j_0-1}$ and $z_{j_0+1}$. Let $L_\pm= \dist(z_{j_0\pm 1},I)>0$. Here if $j_0=n$, then $L_+=\infty$. Let $r_1,\dots,\ha r_{j_0},\dots,r_n$, and $s_1,\dots,s_m$ be as in Lemma \ref{ping-pong} except that we now require that  $r_{j_0\pm 1}<(|z_{j_0\pm 1}-z_{j_0\pm 2}|\wedge L_\pm)/2$.  Let
\begin{align*}
  Q=&L_-^{-2\alpha} \prod_{j=1}^{j_0-2} |z_j-z_{j+1}| \cdot |w_m|^{-\alpha} \prod_{k=1}^{m-1} (|w_k|\wedge |w_k-w_{k+1}|)^{-\alpha}
\\
  &\cdot (L_+\wedge |z_{j_0+1}-z_{j_0+2}|)^{-\alpha} \prod_{j=j_0+2}^n (|z_j-z_{j+1}|\wedge |z_j-z_{j-1}|)^{-\alpha}.
  \end{align*}
Here when $m=0$, the $|w_m|^{-\alpha} \prod_{k=1}^{m-1} (|w_k|\wedge |w_k-w_{k+1}|)^{-\alpha}$ disappears; and when $j_0=n$, the second line in the formula disappears. Let $h\in(L_+\wedge L_-)/2$ and  $\tau^I_{h}=\tau_{I\times [0,h]}$.
Then
\begin{align}
  \PP[\tau^I_{h}<\tau^{z_j}_{r_j}<\infty,j\in\N_n\sem\{j_0\}; \tau^I_{h}<\tau^{w_k}_{s_k}<\infty,k\in\N_m]\lesssim Q h^\alpha \prod_{j\in \N_m\sem \{j_0\}} r_j^\alpha \cdot \prod_{k=1}^m  s_k^\alpha.  \label{ping-pong-ineq-I}
\end{align}
 \label{ping-pong-I}
\end{Lemma}
\begin{proof}
  The proof is similar to that of Lemma \ref{ping-pong}. The only essential difference is that now we do not get an upper bound of $\PP[\tau^I_{h}<\infty]$ using (\ref{1pt}). 
By symmetry we assume that $z_j$'s are positive and $w_k$'s are negative. Let $\tau=\tau^I_{h}$ and $z_{j_0}=\Ree \gamma(\tau)$. Then  $z_{j_0}$ is $\F_\tau$-measurable, and $z_{j_0-1}<z_{j_0}<z_{j_0+1}$.

Let $E$ denote the event in (\ref{ping-pong-ineq-I}). Then $E=E_*^\tau\cap E_\#$, where
  $$E^\tau_*:=\{\tau <\tau^{z_j}_{r_j}\wedge \tau^*_{z_j},j\in\N_n\sem \{j_0\};\tau <\tau^{w_k}_{s_k}\wedge \tau^*_{w_k},1\le k\le m\}\in\F_\tau;$$
  $$E_\#:=\{ \tau^{z_j}_{r_j}<\infty,j\in\N_n\sem \{j_0\}; \tau^{w_k}_{s_k}<\infty,1\le k\le m \}.$$

Suppose $E_*^\tau$ occurs.  Define $\til z_j,D_j,\til D_j,\til r_j,\til w_k,E_k,\til E_k,\til s_k$ as in the previous proof.
By DMP of chordal SLE$_\kappa$ and Proposition \ref{lower-upper}, we see that (\ref{EE0}) also holds here.

The estimates (\ref{mE1},\ref{jD1},\ref{kE1}) still hold here by the same extremal length argument. Estimate (\ref{jD1'}) should be replaced by
\BGE 1\wedge \frac{\til r_{j_0-1}}{|\til z_{j_0-1}|\wedge |\til z_{j_0-1}-\til z_{j_0}|} \lesssim \frac{r_{j_0-1} h }{|z_{j_0-1}-z_{j_0}|^2 } \le  \frac{r_{j_0-1} h }{L_-^2 } . \label{jD1'-I}\EDE
When $j_0=n$, combining (\ref{mE1},\ref{jD1},\ref{kE1},\ref{jD1'-I}) with (\ref{EE0}), we get \begin{align*} \PP[E_\#|\F_\tau,E_*^\tau]\lesssim & \Big(\frac{h r_{j_0-1}}{L_-^2}\Big)^\alpha \Big(\frac{s_m}{|w_m |}\Big)^\alpha \cdot \prod_{j=1}^{j_0-2} \Big(\frac{r_j}{|z_j-z_{j+1}|}\Big)^\alpha \cdot \prod_{k=1}^{m-1} \Big(\frac{s_k}{|w_k-w_{k+1}|\wedge|w_k|}\Big)^\alpha
\end{align*}
Taking expectation, we then get (\ref{ping-pong-ineq-I}) in the case $j_0=n$.

Suppose $2\le j_0\le n-1$. Let $R_j$, $j_0+2\le j\le n$, be as in the proof of Lemma \ref{ping-pong}. We redefine $R_{j_0+1}=(|z_{j_0+1}-z_{j_0+2}|\wedge L_+)/2$.
For each $\ulin k=(k_{j_0+1},\dots,k_n)\in (\N\cup\{0\})^{\N_{(j_0,n]}}$, let $E_{\ulin k}$ be defined as in the proof of Lemma \ref{ping-pong} using the $R_j$, $j_0+1\le j\le n$ defined here. By Remark \ref{after-inner-last},
 \BGE \PP[E_{\ulin k}]\lesssim   \prod_{j=j_0+1}^n  e^{-2\alpha k_j}.\label{PEk-after}\EDE
 This inequality holds no matter whether all $k_j$'s are zero or not.
On the event $E_{\ulin k}\cap E_*$, the same extremal length argument shows that
\BGE 1\wedge \frac{\til r_j}{|\til z_j-\til z_{j-1}|}\lesssim   \frac{r_j}{R_j e^{-k_j}},\quad j_0+1\le j\le n.\label{j-D=>-after}\EDE

Combining (\ref{EE0},\ref{mE1},\ref{jD1},\ref{kE1},\ref{jD1'-I})  with (\ref{PEk-after},\ref{j-D=>-after}) we get
\begin{align*} \PP[E\cap E_{\ulin k}]\lesssim & \Big(\frac{h r_{j_0-1}}{L_-^2}\Big)^\alpha  \prod_{j=1}^{j_0-2} \Big(\frac{r_j}{|z_j-z_{j+1}|}\Big)^\alpha \cdot \prod_{j=j_0+1}^n \Big( \frac{r_j}{R_j }\Big)^\alpha \\
\cdot & \Big(\frac{s_m}{|w_m |}\Big)^\alpha \cdot  \prod_{k=1}^{m-1} \Big(\frac{s_k}{|w_k-w_{k+1}|\wedge|w_k|}\Big)^\alpha \cdot  \prod_{j=j_0+1}^n  e^{-\alpha k_j} 
\end{align*}
Summing up the above inequality over $\ulin k\in (\N\cup\{0\})^{\N_{(j_0,n]}}$, we get (\ref{ping-pong-ineq-I}) for $j_0<n$.
\end{proof}

\begin{Definition}
Recall the $\Sigma_n$, $n\in\N$, defined in Theorem \ref{Main-thm}. For  $\ulin z\in \Sigma_n$ and $j_0\in\N_n$, we say that $z_{j_0}$ is an innermost component of $\ulin z$ if there is no $k\in\N_n\sem \{j_0\}$ such that $z_k$ lies strictly between $0$ and $z_{j_0}$. An element $\ulin z\in\Sigma_n$ may have one or two innermost components. For $\ulin z=(z_1,\dots,z_n)\in \Sigma_n$, we define the inner distance of $\ulin z$ by $d(\ulin z):=\min\{|z_j-z_k|:0\le j<k\le n\}$, where $z_0:=0$. \label{Definition-Sigma}
\end{Definition}

\begin{Lemma}
Let $\ulin z^*=(z_1^*,\dots,z_n^*)\in\Sigma_n$.  Suppose that $z_1^*$ is an innermost component of $\ulin z^*$. Then for any $\eps>0$, there are $\delta\in(0,d(z^*)/3]$ and an $\HH$-hull $H$ (depending on $\ulin z^*$ and $\eps$) such that
 \begin{itemize}
   \item $\{z\in\HH:|z-z_1^*|\le 3\delta\}\subset H$;
   \item $\dist(z_j^*,H)\ge 3\delta$, $2\le j\le n$; and
   \item if $\ulin z\in\Sigma_n$ and $\ulin r\in (0,\infty)^n$ satisfy $\Vert z-z^*\Vert_\infty\le \delta$ and $\Vert \ulin r\Vert_\infty\le \delta$, then
\BGE \PP[K_{\tau^{z_1}_{r_1}} \not\subset H;\tau^{z_1}_{r_1}<\tau^{z_j}_{r_j} <\infty,2\le j\le n]<\eps  \prod_{j=1}^n r_j^\alpha .\label{compact-lem-ineq}\EDE
\end{itemize}
\label{compact-lem}
\end{Lemma}
\begin{proof}
For $\ulin r=(r_1,\dots,r_n)\in(0,\infty)^n$, let $P(\ulin r)=\prod_{j=1}^n r_j$. Fix a chordal SLE$_\kappa$ curve $\gamma$ in $\HH$ from $0$ to $\infty$. For $\ulin z=(z_1,\dots,z_n)\in\Sigma_n$, $\ulin r=(r_1,\dots,r_n)\in(0,\infty)^n$, and  $S\subset\HH$, let
$$E^{\ulin z}_{\ulin r;S}=\{\tau^{z_1}_{r_1}<\tau^{z_j}_{r_j} <\infty,2\le j\le n;K_{\tau^{z_1}_{r_1}}\cap S\ne \emptyset \} .$$
Then 
(\ref{compact-lem-ineq}) can be rewritten as $\PP[ E^{\ulin z}_{\ulin r;\HH\sem H}]< \eps P(\ulin r)$.

By Lemma \ref{Lemma-PR}, there is a positive continuous functions  $F_\infty$ on $\Sigma_n$ such that, for any $\ulin z\in\Sigma_n$ and any $\ulin r\in(0,\infty)^n$,
\BGE \PP[E^{\ulin z}_{\ulin r; \{z\in\HH:|z|\ge R\}}]\le F_\infty(\ulin z)R^{-\alpha}  P(\ulin r),\quad \mbox{if }\Vert \ulin r\Vert_\infty<d(\ulin z)/2\mbox{ and } R\ge 2\max\{|z_k|\}.\label{R-infty}\EDE
By Proposition \ref{RZ-Thm3.1}, for any $2\le k\le n$, there are a constant $\beta>0$ and a positive continuous function $F_k$ on $\Sigma_n$ such that, for any $\ulin z\in\Sigma_n$, $\ulin r\in(0,\infty)^n$, and $r>0$,
\BGE \PP[E^{\ulin z}_{\ulin r; \{z\in\HH:|z-z_k|\le r\}}]\le  F_k(\ulin z) r^\beta P(\ulin r),\quad \mbox{if }\Vert \ulin r\Vert_\infty<d(\ulin z)/8.\label{rk}\EDE


Note that, if $\Vert \ulin z-\ulin z^*\Vert_\infty\le d(\ulin z^*)/4$, then $d(\ulin z)\ge d(\ulin z^*)/2$ and $\max\{|z_j|\}\le 2 \max\{|z_j^*|\}$.
By (\ref{R-infty},\ref{rk}) and the continuity of $F_\infty$ and $F_k$, $2\le k\le n$,  there are $R>4\max\{|z_k^*|\}$ and $r\in (0,d(\ulin z^*)/3)$ such that if  $\Vert \ulin z-\ulin z^*\Vert_\infty\le d(\ulin z^*)/4$, and $\Vert \ulin r\Vert_\infty<d(\ulin z^*)/16$, then
$$ \PP[E^{\ulin z}_{\ulin r; \{z\in\HH:|z|\ge R\}\cup \bigcup_{k=2}^n\{z\in\HH:|z-z_k|\le r\}  }] <\frac \eps 2 P[\ulin r]. $$ 
We further assume that $\Vert \ulin z-\ulin z^*\Vert_\infty\le r/2$. Then $\{z\in\HH:|z-z_k^*|\le r/2\}\subset \{z\in\HH:|z-z_k|\le r\}$ for $2\le k\le n$, which implies by the above formula that
 \BGE \PP[E^{\ulin z}_{\ulin r; \{z\in\HH:|z|\ge R\}\cup \bigcup_{k=2}^n\{z\in\HH:|z-z_k^*|\le r/2\}  }] <\frac \eps 2 P[\ulin r],\quad \mbox{if }\Vert \ulin r\Vert_\infty<d(\ulin z^*)/16. \label{Rinftyrk}\EDE
Since $R>2\max\{|z_k|\}$ and $r<d(\ulin z^*)/3$, the semi-discs  $\{z\in\HH:|z-z_j^*|\le r\}$, $1\le j\le n$, are mutually disjoint, and are all contained in the semi-disc $\{z\in\HH:|z|\le R\}$.

By symmetry, we assume that $z_1^*>0$. We relabel the components of $\ulin z^*$ by $z_j^*$, $1\le j\le n'$, and $w_k^*$, $1\le k\le m'$, where $n'\ge 1$, $m'\ge 0$, and $n'+m'=n$, such that  $w_{m'}^*<\cdots<w_1^*<0<z_1^*<\dots<z_{n'}^*$. After relabeling, the symbol $z_1^*$ still refers to the same point. Correspondingly, we relabel the components of every $\ulin z\in\Sigma_n$ and $\ulin r\in(0,\infty)^n$  by $z_j$, $1\le j\le n'$, $w_k$, $1\le k\le m'$, $r_j$, $1\le j\le n'$, and $s_k$, $1\le k\le m'$. It is clear that, if  $\Vert \ulin z-\ulin z^*\Vert_\infty<d(z^*)/2$, then  $w_{m'}<\cdots<w_1<0<z_1<\dots<z_{n'}$, and so $z_1$ is an innermost component of $\ulin z$.

Define compact intervals $I_j$, $2\le j\le n$, and $J_k$, $1\le k\le m$, as follows. If $n'=1$, we do not define $I_j$'s. If $n'\ge 2$, let $I_{n'}=[z_{n'}^*+r/2,R ]$, and $I_j=[z_j^*+r/2,z_{j+1}^*-r/2]$, $2\le j\le n'-1$. If $m=0$, we do not define $J_k$'s. If $m\ge 1$, let $J_{m'}=[-R ,w_{m'}^*-r/2]$, and $J_k=[w_{k+1}^*+r/2,w_k^*-r/2]$, $1\le k\le m'-1$. 

If $\Vert \ulin z-\ulin z^*\Vert_\infty\le r/4$, then the distance from every component of $\ulin z$ to every interval $I_j$ or $J_k$ is at least $r/4$. By Lemma \ref{ping-pong-I},  there are continuous functions $F_{I_j}$, $2\le j\le n'$, and $F_{J_k}$, $1\le k\le m'$, defined on the set of $\ulin z\in\Sigma_n$ with $\Vert \ulin z-\ulin z^*\Vert_\infty\le r/4$, such that, if $\Vert \ulin z-\ulin z^*\Vert_\infty\le r/4$, $\Vert \ulin r\Vert_\infty< r/8$, and $h<r/8$, then for each $2\le j\le n'$ and $1\le k\le m'$,
$$ \PP[E^{\ulin z}_{\ulin r; I_j\times [0,h]  }]< F_{I_j}(\ulin z) h^\alpha P(\ulin r),\quad \PP[E^{\ulin z}_{\ulin r; J_k\times [0,h]  }]< F_{J_k}(\ulin z) h^\alpha P(\ulin r).$$
Thus, there is $h>0$ such that, if $\Vert \ulin z-\ulin z^*\Vert_\infty\le r/4$ and $\Vert \ulin r\Vert_\infty\le r/8$, then
 \BGE \PP[E^{\ulin z}_{\ulin r; I_j\times [0,h]  }],\PP[E^{\ulin z}_{\ulin r; J_k\times [0,h]  }] <\frac \eps {2n} P(\ulin r),\quad 2\le j\le n',\quad 1\le k\le m'. \label{IJh}\EDE

 Let $$H=\{z\in\HH:|z|\le R\}\sem \bigcup_{j=2}^{n'}(\{|z-z_j^*|\le r\}\cup I_j\times [0,h])\sem \bigcup_{k=1}^{m'} (\{|z-w_k^*|\le r\}\cup J_k\times [0,h]).$$
See Figure \ref{H}. Then $H$ is an $\HH$-hull, which contains $\{z\in\HH:|z-z_1^*|\le r\}$, and the distance from each of $z_2^*,\dots,z_{n'}^*$ and $w_1^*,\dots,w_{m'}^*$ to $H$ is at least $r$.  Combining (\ref{Rinftyrk}) and (\ref{IJh}), we get  $\PP[ E^{\ulin z}_{\ulin r;\HH\sem H}]< \eps P(\ulin r)$ if $\Vert \ulin z-\ulin z^*\Vert_\infty\le r/4$, and $\Vert \ulin r\Vert_\infty\le r/16$. So we find that (\ref{compact-lem-ineq}) holds for such $H$ and $\delta:=r/16$.
\end{proof}

\begin{figure}
\centering
\includegraphics[width=1\textwidth]{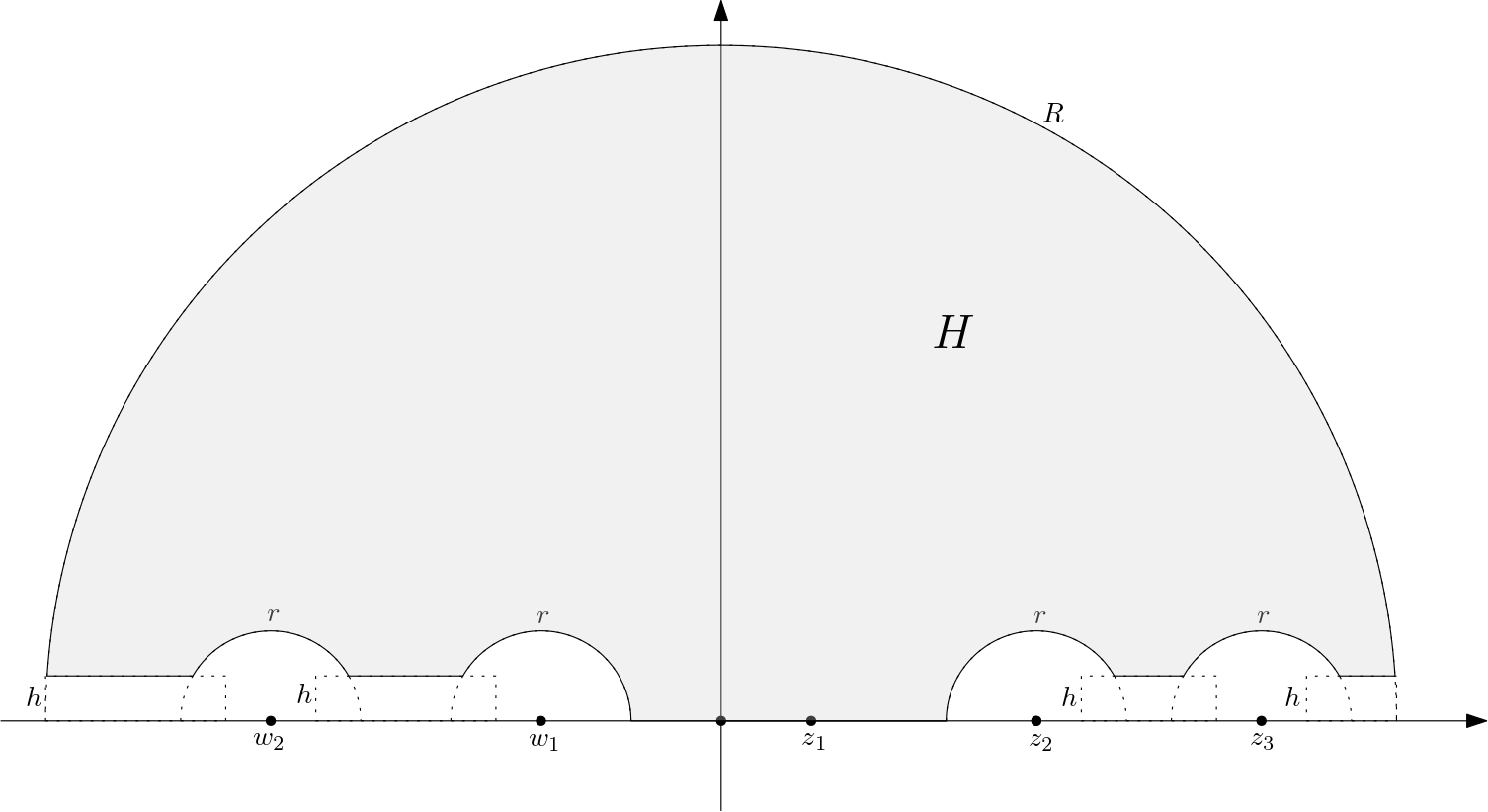}
\caption[A figure for the proof of Lemma \ref{compact-lem}]{\textbf{A figure for the proof of Lemma \ref{compact-lem}.} This figure illustrates the construction of the $\HH$-hull $H$ in the proof of Lemma \ref{compact-lem} in the case that $n'=3$ and $m'=2$.  The $\HH$-hull $H$ (the shaded region) is obtained by removing    small discs of radius $r$ centered at $z_2,z_3,w_1,w_2$ and $4$ rectangles with real interval bases and height $h$ from the big semi-disc $\{z\in\HH:|z|\le R\}$. }
\label{H}
\end{figure}

\section{Proof of the Main Theorem}\label{Chap-main-thm}
We will finish the proof of Theorem \ref{Main-thm} in this section. Recall that  $\PP$  denotes the law of a chordal SLE$_\kappa$ curve in $\HH$ from $0$ to $\infty$; and for $z\in\R\sem\{0\}$ and $r>0$, $\PP_z^*$ denotes the law of a two-sided chordal SLE$_\kappa$ curve in $\HH$ from $0$ to $\infty$ passing through $z$, and $\PP_z^r$ denotes the conditional law $\PP[\cdot|\tau^z_r<\infty]$.

We will use an induction on $n$. By (\ref{G(z)-approx}), Theorem \ref{Main-thm} holds for $n=1$. Let $n\ge 2$. We make the induction hypothesis that Theorem \ref{Main-thm} holds for $n-1$. For any $\ulin w=(w_1,\dots,w_{n-1})\in \Sigma_{n-1}$ (Definition  and $\ulin s=(s_1,\dots,s_{n-1})\in (0,\infty)^{n-1}$, we define
\BGE G(\ulin w,\ulin s)=\PP[\tau^{w_j}_{s_j}<\infty,1\le j\le n].\label{Gws}\EDE
By the induction hypothesis,
$\lim_{s_1,\dots,s_{n-1}\to 0^+} \prod_{j=1}^{n-1} s_j ^{-\alpha} G(\ulin w,\ulin s)=G(\ulin w)$.

Given a chordal Loewner curve $\gamma$ with the corresponding centered Loewner maps $Z_t$'s, we define a family of functions $G^\gamma_t$, $t\ge 0$, on $\Sigma_{n-1}$ associated with $\gamma$ by
\BGE
  G_t^\gamma(z_2,\dots,z_{n})=\left\{\begin{array}{ll}
 \prod_{j=2}^{n} |Z_t'(z_j)|^{\alpha}  G(Z_t(z_2) ,\dots,Z_t(z_{n})),&\mbox{if }t<\tau^*_{z_j}, 2\le j\le n;\\
 0,&\mbox{otherwise.}
 \end{array}
 \right.
\label{vvGt}
\EDE
When $\gamma$ is a random Loewner curve, $G_t^\gamma$ are random functions.  We use
$\EE_{z_1}^*[G_{T_{z_1}}(\cdot)]$
to denote the expectation of $G^\gamma_t(\cdot)$ when $\gamma$ follows the law $\PP^*_{z_1}$, and $t=T_{z_1}$.

Following the approach in \cite{LW}, we will prove that for any $1\le j_0\le n$ and $\ulin z=(z_1,\dots,z_n)\in \Sigma_n$, the following limit exists and is finite:
\BGE G^{j_0}(\ulin z):=\lim_{r_1,\dots,r_n\to 0^+} \prod_{j=1}^n r_j^{-\alpha} \PP[\tau^{z_{j_0}}_{r_{j_0}}<\tau^{z_k}_{r_k}<\infty,\forall k\in\N_n\sem\{j_0\}].\label{ordered}\EDE
It is clear that if the above limit exists and is finite for any $1\le j_0\le n$, then the same is true for the limit in (\ref{n-pt Green}), and we have
\BGE G(\ulin z)=\sum_{j=1}^n   G^j(\ulin z).\label{G-sum}\EDE
In this section we will prove the following theorem.

\begin{Theorem}
Given the induction hypothesis, for any $1\le j_0\le n$, the limit  in (\ref{ordered})  converges uniformly on any compact subset of $\Sigma_n$, and the limit function $G^{j_0}$ is continuous on $\Sigma_n$. Moreover, we have
\BGE  G^{j_0} (z_1,\dots,z_n)=G(z_{j_0})\EE_{z_{j_0}}^*[ G_{T_{z_{j_0}}}(z_1,\dots,\ha z_{j_0},\dots,z_n)],\label{induction}\EDE
 where the symbol $\ha z_{j_0}$ means that $z_{j_0}$ is omitted in the list from $z_1$ to $z_n$.
\label{main}
\end{Theorem}

It is clear that all statements of Theorem \ref{Main-thm} in the induction step except for $G\asymp F$ follow from Theorem \ref{main} and (\ref{G-sum}). When we have the existence of $G$ on $\Sigma_n$, the statement $G\asymp F$ then follows immediately from Proposition \ref{lower-upper} by sending $r_1,\dots,r_n$ to $0^+$.

After proving Theorem \ref{main}, we get a local martingale related to the Green's function.

\begin{Corollary} \label{martingale}
For any fixed $\ulin z=(z_1,\dots,z_n)\in \Sigma_n$, the process
$t\mapsto  G^\gamma_t(\ulin z)$ associated with a chordal SLE$_\kappa$ curve $\gamma$ in $\HH$ from $0$ to $\infty$ is a local martingale up to $\tau:=\min\{\tau^*_{z_j},1\le j\le n\}$.
\end{Corollary}
\begin{proof}
Fix $\ulin z=(z_1,\dots,z_n)\in \Sigma_n$ and let $M_t= G_t^\gamma(\ulin z)$.
	It suffices to prove that for any $\HH$-hull $K$, whose  closure does not contain any of $z_1,\dots,z_n$, $M_{\cdot\wedge T_K}$ is a martingale, where  $T_K:=\inf\{t>0:\gamma[0,t]\not\subset \lin K\}$.  The reason is that $\tau$ is the supremum of all such $T_K$.  To prove that $M_{\cdot\wedge T_K}$ is a martingale, we pick a small $r>0$, and consider the martingale
	$$M^{(r)}_t:=r^{-n\alpha }\PP[\tau^{z_j}_r< \infty,1\le j\le n|\F_{t}].$$
	By Theorem \ref{main}, DMP of chordal SLE and Koebe's distortion theorem, we have $M^{(r)}_t\to M_{t }$ on $[0,\tau)$ as $r\to 0^+$. We claim that the convergence is uniform on  $[0,T_K]$.  To see this, we apply Proposition \ref{compact-prop} to conclude that there exist an $\HH$-hull $H$ and $a<b\in(0,\infty)$ such that $(Z_t(z_1),\dots,Z_t(z_n))\in H$ and $a\le |Z_t'(z_j)|\le b$, $1\le j\le n$, for any $t\in[0,T_K]$. So we get the uniform convergence of $M^{(r)}_t\to M_{t }$ over $[0,T_K]$ by the uniform convergence of the $n$-point Green's function on the  $\HH$-hull in $H$. So the claim is proved, which then implies that   $M_{\cdot\wedge T_K}$ is a martingale, as desired.
\end{proof}

\begin{Remark}
  We may write $M_t=\prod_{j=1}^n|g_t'(z_j)|^{\alpha}  G(g_t(z_1)-U_t,\dots,g_t(z_n)-U_t)$.
If we know that $ G$ is $C^2$, then using It\^o's formula and Loewner's equation (\ref{chordal}), one can easily get the following second order PDE for $ G$:
$$\frac{\kappa}{2}\Big(\sum_{j=1}^n \pa_{z_j}\Big)^2  G
+\sum_{j=1}^n \pa_{z_j}\frac{2}{z_j }\cdot   G
+\alpha  \sum_{j=1}^n \frac{-2 }{z_j^2}\cdot  G=0.$$
Since the PDE does not depend on the order of points, it is also satisfied by the unordered Green's function $G$.

We expect that the smoothness of $ G$ can be proved by H\"ormander's theorem because the differential operator in the above displayed formula satisfies H\"ormander's condition.
\end{Remark}

The rest part of the paper is devoted to the proof of Theorem \ref{main}. By symmetry it suffices to work on the case $j_0=1$.  We will first prove in Section \ref{proof1} the existence of $G^1$ as well as the uniform convergence on compact subsets of $\Sigma_n$, and then prove in Section \ref{proof2} the continuity of $G^1$.

\subsection{Existence}\label{proof1}
In this subsection, we work on the inductive step to prove the existence of the limit in (\ref{ordered}) with $j_0=1$. We now define $G^1$ on $\Sigma_n$ using (\ref{induction}) instead of (\ref{ordered}).
In order to prove that the limit in (\ref{ordered}) converges uniformly on each compact subset of $\Sigma_n$, it suffices to show that, for any $\ulin z^*=(z_1^*,\dots,z_n^*)\in\Sigma_n$ and $\eps>0$, there exists $\delta>0$ such that if $\ulin z=(z_1,\dots,z_n)\in \Sigma_n$ and $\ulin r=(r_1,\dots,r_n)\in(0,\infty)^n$ satisfy that $\Vert \ulin z-\ulin z^*\Vert_\infty<\delta$ and $\Vert \ulin r\Vert_\infty<\delta$, then
\BGE |\prod_{j=1}^n r_j^{-\alpha} \PP[\tau^{z_1}_{r_1}< \tau^{z_j}_{r_j}<\infty,2\le j\le n]- G^1(z_1,\dots,z_n)|<\eps. \label{local-uniform}\EDE

Fix $\ulin {z}^*=(z_1^*,\dots,z_n^*)\in\Sigma_n$ and $\eps>0$. Recall Definition \ref{Definition-Sigma}. Let $d^*=d(\ulin z^*)$.
Let $\ulin z=(z_1,\dots,z_n)\in\Sigma_n$ satisfy $\Vert \ulin z-\ulin z^*\Vert_\infty< d^*/2$. 
First suppose $z_1^*$ is not an innermost component of $\ulin z^*$. Then $z_1$ is not an innermost component of $\ulin z$. Then there is $k_0\in\{2,\dots,n\}$ such that $z_{k_0}$ lies strictly between $0$ and $z_1$. Under the law $\PP_{z_1}^*$, we have $\tau^*_{z_{k_0}}\le T_{z_1}$, and so $ G_{T_{z_1}}(z_2,\dots,z_n)=0$, which implies that $G^1(\ulin z)=0$.  On the other hand, by Lemma \ref{ping-pong},
$$\lim_{r_1,\dots,r_n\to 0^+} \prod_{j=1}^n r_j^{-\alpha} \PP[\tau^{z_1}_{r_1}<\tau^{z_k}_{r_k}<\infty,2\le k\le n]=0,$$
and the convergence is uniform  in some neighborhood of $\ulin z^*$. So we have (\ref{local-uniform}) if $z_1^*$ is not an innermost component of $\ulin z^*$.
From now on, we assume that $z_1^*$ is an innermost component of $\ulin z^*$. 
By symmetry we assume that $z_1^*>0$.



Let $\ulin z=(z_1,\dots,z_n)\in\Sigma_n$ and  $\ulin r=(r_1,\dots,r_n)\in (0,\infty)^n $. Suppose $\Vert \ulin z-\ulin z^*\Vert_\infty< d^*/4$ and $\Vert \ulin r\Vert_\infty<d^*/4$. Then the discs $\{|z-z_j|\le r_j\}$, $1\le j\le n$, are mutually disjoint.
Let $E_{\ulin r}$ denote the event $\{\tau^{z_1}_{r_1}<\tau^{{z_j}}_{r_{j}}<\infty,2\le j\le n\}$.
We will transform the rescaled probability $\prod_{j=1}^n r_j^{-\alpha} \PP[E_{\ulin r}]$ into  $G^1(\ulin z)$ (defined by (\ref{induction})) in a number of steps. In each step we get an error term, and have an upper bound of the error term. 

We define some good events depending on $\ulin z$.
For any $r>0$ and $\HH$-hull $H$, let $E_{r;H}$ denote the event that $K_{\tau^{z_1}_r}\subset H$.
For $R>r>s\ge 0$, let $E_{r,s;R}$ be the event that $\gamma[\tau^{z_1}_r,\tau^{z_1}_s]$ does not intersect the connected component of $\{z\in\HH:|z-z_1|=R\}\cap H_{\tau^{z_1}_{r}}$ which has $z_1+R$  as an endpoint.

In the following, we use $X\st{e}{\approx}Y$ to denote the approximation relation $|X-Y|=e$, and call $e$ the error term. Let $\ulin z'=(z_2,\dots,z_n)$, $\ulin r'=(r_2,\dots,r_n)$, and $E_{\ulin r'}=\{\tau^{z_j}_{r_j} <\infty,2\le j\le n\}$. 
For some $\HH$-hull $H$ to be determined we use the following approximation relations:
\begin{align*}
 & \PP[E_{\ulin r}]\st{e_1^*}{\approx} \PP[E_{\ulin r}\cap E_{r_1;H}]
  =   \PP[\tau^{z_1}_{r_1}<\infty]\cdot  \EE^{r_1}_{z_1}[{\mathbf 1}_{E_{r_1;H} }\PP[E_{\ulin r'}| \F_{\tau^{z_1}_{r_1}}]]\\
  \st{e_2^*}{\approx} & r_1^{\alpha}  G(z_1) \EE^{r_1}_{z_1}[{\mathbf 1}_{E_{r_1;H}}\PP[E_{\ulin r'}| \F_{\tau^{z_1}_{r_1}}] ]
  \st{e_3^*}{\approx}  G(z_1)  \EE^{r_1}_{z_1}[{\mathbf 1}_{E_{r_1;H}} G_{\tau^{z_1}_{r_1}}(\ulin z')]  \prod_{k=1}^n r_k^{\alpha}.
\end{align*}
We write $G(r,\cdot)$ for $G_{\tau^{z_1}_r}$.   For some $\eta_2>\eta_1>r_1$ to be determined, we further use the following approximation relations:
\begin{align*}
  & G(z_1)\EE_{z_1}^{{r_1}}[{\mathbf 1}_{E_{r_1;H}}  G_{}({r_1},\ulin z')]
   \st{e_4}{\approx}
  G(z_1)\EE_{z_1}^{{r_1}}[{\mathbf 1}_{E_{r_1;H}\cap E_{\eta_1,r_1;\eta_2}}  G_{ }({r_1},\ulin z')]\\
\st{e_5}{\approx} &   G(z_1)\EE_{z_1}^{{r_1}}[{\mathbf 1}_{E_{\eta_1;H}\cap E_{\eta_1,{r_1};\eta_2}}  G_{ }({r_1},\ulin z')] \st{e_6}{\approx}     G(z_1)\EE_{z_1}^{{r_1}}[{\mathbf 1}_{E_{\eta_1;H}\cap E_{\eta_1,{r_1};\eta_2}}  G_{ }({\eta_1},\ulin z')]\\
\st{e_{7}}{\approx}   &  G(z_1)\EE_{z_1}^{{r_1}}[{\mathbf 1}_{E_{\eta_1;H} }  G_{ }({\eta_1},\ulin z')] 
\st{e_{8}}{\approx} G(z_1)\EE^*_{z_1}[ {\mathbf 1}_{E_{\eta_1;H}}  G_{ }({\eta_1},\ulin z')]
\\\st{e_{9}}{\approx} & G(z_1)\EE^*_{z_1}[{\mathbf 1}_{E_{\eta_1;H}\cap E_{\eta_1,0;\eta_2}}   G_{ }({\eta_1},\ulin z')]
\st {e_{10}}{\approx}  G(z_1)\EE^*_{z_1}[{\mathbf 1}_{E_{\eta_1;H}\cap E_{\eta_1,0;\eta_2}}   G_{ }(0 ,\ulin z')] \\ \st {e_{11}}{\approx}  & G(z_1)\EE^*_{z_1}[{\mathbf 1}_{E_{0;H}\cap E_{\eta_1,0;\eta_2}}   G_{ }(0 ,\ulin z')]
 \st {e_{12}}{\approx} G(z_1)\EE^*_{z_1}[{\mathbf 1}_{E_{0;H}}   G_{}(0,\ulin z')] \\ \st {e_{13}}{\approx} &  G(z_1)\EE^*_{z_1}[   G_{ }(0,\ulin z')]  = G(\ulin z).
\end{align*}
Let $e_j=e_j^*/\prod_{k=1}^n r_k$, $j=1,2,3$. Then
\BGE \Big|\prod_{k=1}^n r_k^{-\alpha}  \PP[E_{\ulin r}] -  G(\ulin z)\Big|\le \sum_{j=1}^{12} e_j.\label{sum-ej}\EDE

Let $\tau=\tau^{z_1}_{r_1}$, $D_j=\{z\in H_\tau:|z-z_j|\le r_j\}$, $\til z_j=Z_\tau(z_j)$, $\til D_j=Z_\tau(D_j)$, $\til r_j^+=\rad_{\til z_j}(\til D_j)$ and $\til r_j^-=\dist(\til z_j,\pa \til D_j\cap \HH)$, $2\le j\le n$. Let $\ulin{\til z}'=(\til z_2,\dots,\til z_n)$ and $\ulin {\til r}'_\pm=(\til r_2^\pm,\dots,\til r_n^\pm)$. By Koebe distortion theorem, for any $2\le j\le n$, if $r_j<\dist(z_j,K_\tau)$,
\BGE \frac{|Z_\tau'(z_j)| r_j}{(1+r_j/\dist(z_j,K_\tau))^2}\le \til r_j^-\le \til r_j^+\le \frac{|Z_\tau'(z_j)| r_j}{(1-r_j/\dist(z_j,K_\tau))^2} .\label{distortion}\EDE
By DMP of chordal SLE and (\ref{Gws}),
\BGE G(\ulin{\til z}',\ulin{\til r}'_-)\le  \PP[E_{\ulin r'}|\F_\tau,\tau<\tau^{z_j}_{r_j},2\le j\le n]\le  G(\ulin{\til z}',\ulin{\til r}'_+).\label{DMP}\EDE

Let $C_\kappa\in[1,\infty)$ be the constant in Proposition \ref{RN<1}.
By Lemma \ref{compact-lem}, there are a nonempty $\HH$-hull $H$ and $\delta_H\in (0, d^*/3]$, such that $\{z\in\HH:|z-z_1^*|\le 3\delta_H\}\subset H$, $\dist(z_j,H)\ge 3\delta_H$, $2\le j\le n$,  and whenever $\Vert\ulin z-\ulin z^*\Vert_\infty\le \delta_H$ and $\Vert \ulin r\Vert _\infty\le
\delta_H$, we have
\BGE \prod_{j=1}^n r_j^{-\alpha} \PP[E_{r_1;H}^c\cap  E_{\ulin r}]<\frac\eps{11 C_\kappa } .\label{est-H}\EDE

From now on, we always assume that $\Vert\ulin z-\ulin z^*\Vert_\infty<\delta_H$. Then $H\supset \{z\in\HH:|z-z_1|\le  2\delta_H\}$, and $\dist(z_j,H)\ge 2\delta_H$, $2\le j\le n$. By (\ref{est-H}), if $\Vert \ulin r\Vert_\infty\le \delta_H$,
$$e_1\le \frac\eps{11 C_\kappa }.$$

Sending  $r_2,\dots,r_n$ to $0^+$ in (\ref{est-H}) and using Fatou's lemma, estimates (\ref{distortion},\ref{DMP}) and the convergence of $(n-1)$-point Green's function, we get
$$ r_1^{-\alpha} \PP[\tau^{z_1}_{r_1}<\infty] \cdot \EE_{z_1}^{r_1}[{\mathbf 1}_{E_{r_1;H}^c}  G (r_1,\ulin z')]\le \frac \eps {11 C_\kappa},\quad \mbox{if }r_1\le \delta_H.$$
By Proposition \ref{RN<1}, if $r_1\le \delta_H$,
$$ r_1^{-\alpha} \PP[\tau^{z_1}_{r_1}<\infty] \cdot \EE_{z_1}^{*}[{\mathbf 1}_{  E_{r_1;H}^c}  G (r_1,\ulin z')]\le   \frac \eps {11}.$$
Let $r_1\to 0^+$. From $ r_1^{-\alpha} \PP[\tau^{z_1}_{r_1}<\infty]\to G(z_1)$, $E_{r_1;H}^c\to E_{0;H}^c$,  $G(r_1,\ulin z)\to G(0,\ulin z)$ and Fatou's lemma, we get
$G(z_1) \EE_{z_1}^{*}[{\mathbf 1}_{  E_{0;H}^c}  G (0,\ulin z')]\le  \frac \eps{11}$, which implies
$$e_{13}\le     \frac \eps{11}.$$

By Proposition \ref{compact-prop},  the set
$$ \Omega_H:=\{(g_K(z_2)-u,\dots,g_K(z_n)-u): K\in{\mathcal H}(H),u\in S_H,|z_j-z_j^*|\le \delta_H,2\le j\le n\}$$ 
is a compact subset of $\Sigma_{n-1}$, and the set
$$  Q_H:= \{|g_K'(z)|:K\in{\mathcal H}(H),z\in\bigcup_{j=2}^n [z_j^*-\delta_H,z_j^*+\delta_H]\}$$ 
is a compact subset of $(0,\infty)$. Let $\xi_H=\min\{|w_k|:\ulin w=(w_2,\dots,w_n)\in \Omega_H,2\le k\le n\}>0$.
For $a\ge 0$, we write $\ulin Z_a(\ulin z')$ for $(Z_{\tau^{z_1}_a}(z_2),\dots,Z_{\tau^{z_1}_a}(z_n))$, when all components are well defined. We will use the fact that, on the event $E_{a;H}$, $\ulin Z_a(\ulin z')\in \Omega_H$ because $Z_{\tau^{z_1}_a}=g_{K_{\tau^{z_1}_a}}-U_{\tau^{z_1}_a}$, $K_{\tau^{z_1}_a}\subset H$, and $U_{\tau^{z_1}_a}\in S_{K_{\tau^{z_1}_a}}\subset H$ by Corollary \ref{UinSK} and Proposition \ref{SKSH}. Recall that we assume that  $\Vert\ulin z-\ulin z^*\Vert_\infty\le \delta_H$. So we have
\BGE |Z_{\tau^{z_1}_a}(z_j)|\ge \xi_H,\quad 2\le j\le n,\mbox{ on the event }E_{a;H}.\label{xiH}\EDE
By the continuity of $(n-1)$-point Green's function and the compactness of $\Omega_H$ and $Q_H$, we see that, for any $a\ge 0$, $G(a,\ulin z')$ is bounded by a constant depending only on $\kappa,n,\ulin z^*,H,\delta_H$ on the event $E_{a;H}$.

By (\ref{distortion},\ref{DMP}), the compactness of $\Omega_H$ and $Q_H$, and Proposition \ref{lower-upper},
 $\PP[E_{\ulin r'}|\F_{\tau^{z_1}_{r_1}},E_{r_1;H}]$ is bounded by $\prod_{j=2}^n r_j^\alpha$ times some constant depending only on $\kappa,n,\ulin z^*,H,\delta_H$. 
By (\ref{G(z)-approx}) and the above bound, there are $\beta_1\in(0,\infty)$ depending only on $\kappa$ and  $C^H_1\in(0,\infty)$ depending only on $\kappa,n,\ulin z^*,H,\delta_H$ such that, if $r_1<|z_1|$,
 $$e_2\le C^H_1 r_1^{\beta_1}.$$

 Since the convergence of $(n-1)$-point ordered Green's function is uniform over compact sets, by (\ref{distortion},\ref{DMP}) and the compactness of $\Omega_H$ and $Q_H$, we find that there is $\delta_H'\in(0,\delta_H)$ depending only on  $\kappa,n,H,\delta_H$ such that if $\Vert\ulin r\Vert_\infty\le \delta_H'$, then
 $$e_3<\frac\eps {11}.$$

Since $G$ is continuous on $\Sigma_{n-1}$, by the compactness of $\Omega_H$ and $Q_H$, $G(a,\ulin z')$ is bounded by some constant depending only on $\kappa,n,\ulin z^*, H,\delta_H$ on the event $E_{a;H}$. Combining this fact for $a\in\{r_1,\eta_1,0\}$ with Proposition \ref{stayin} and the boundedness of $G(z_1)$ (over $[z_1^*-\delta_H,z_1^*+\delta_H]$), we find that there is   $C^H_2\in(0,\infty)$ depending only on $\kappa,n,\ulin z^*,H,\delta_H$ such that
$$e_4,e_7,e_9,e_{12}\le C^H_2 (\eta_1/\eta_2)^\alpha.$$

Since $H\supset \{z\in\HH:|z-z_1^*|\ge 3\delta_H\}$, if  $\eta_2\le 2\delta_H$, then $E_{r_1;H}\cap E_{\eta_1,r_1;\eta_2}=E_{\eta_1;H}\cap E_{\eta_1,r_1;\eta_2}$ and $E_{\eta_1;H}\cap E_{\eta_1,0;\eta_2}=E_{0;H}\cap E_{\eta_1,0;\eta_2}$, which implies that
$$e_5=e_{11}=0.$$

Combining Proposition \ref{Prop2.13} with the  boundedness of $G(z_1)$ and $G(\eta_1,\ulin z')$ on the event $E_{\eta_1;H}$, we find that there are $\beta_2>0$ depending only on $\kappa$ and $C^H_3\in(0,\infty)$ depending only on $\kappa,n,\ulin z^*,H,\delta_H$ such that, if $r_1<\eta_1/6$,
$$e_8\le C_3^H (r_1/\eta_1)^{\beta_2}.$$

Recall that
$$G(\eta_1,\ulin z')=\prod_{j=2}^n |Z_{\tau^{z_1}_{\eta_1}}'(z_j)|^\alpha \cdot G(\ulin Z_{ {\eta_1}}(\ulin z')),\quad G(r_1,\ulin z')= \prod_{j=2}^n |Z_{\tau^{z_1}_{r_1}}'(z_j)|^\alpha\cdot G(\ulin Z_{ {r_1}}(\ulin z')).$$
Assume $\eta_2\le 2\delta_h$. Then   $E_{r_1;H}\cap E_{\eta_1,r_1;\eta_2}=E_{\eta_1;H}\cap E_{\eta_1,r_1;\eta_2}$, and on this common event, $\ulin Z_{{\eta_1}}(\ulin z'),\ulin Z_{{r_1}}(\ulin z')\in \Omega_H$.
Let $K_\Delta= K_{\tau^{z_1}_{r_1}}/ K_{\tau^{z_1}_{\eta_1}}$. 
On the event $E_{\eta_1,r_1;\eta_2}$, by Proposition \ref{lem-extremal2'},  $\diam(K_\Delta)\le 8  \eta_2$, and
by Proposition \ref{centered-Delta}, we have
\BGE \Vert \ulin Z_{\eta_1}(\ulin z')-\ulin Z_{r_1}(\ulin z')\Vert_\infty\le 56 \eta_2 .
\label{Z-r-eta}\EDE
From $K_\Delta=K_{\tau^{z_1}_{r_1}}/ K_{\tau^{z_1}_{\eta_1}}$ we know $g_{\tau^{z_1}_{r_1}}=g_{K_\Delta}\circ g_{\tau^{z_1}_{\eta_1}}$. Let $Z_\Delta=g_{K_\Delta}(\cdot+U_{\tau^{z_1}_{\eta_1}})-U_{\tau^{z_1}_{r_1}}$. Then $Z_{\tau^{z_1}_{r_1}}=Z_\Delta\circ Z_{\tau^{z_1}_{\eta_1}}$ and $Z_\Delta'(z)=g_{K_\Delta}'(\cdot+U_{\tau^{z_1}_{\eta_1}})$. By Proposition \ref{K/U}, $U_ {\tau^{z_1}_{\eta_1}}\in \lin{K_\Delta}$. By Proposition \ref{small}, for $z\in\lin\HH$,
\BGE |Z_\Delta'(z)-1|\le 5 \Big(\frac{8\eta_2}{|z |}\Big)^2,\quad \mbox{if }|z |\ge 40\eta_2.\label{Z-r-eta'}\EDE
Let $\til z_j=Z_ {\tau^{z_1}_{\eta_1}}(z_j)$, $2\le j\le n$. Then  $Z_{\tau^{z_1}_{r_1}}'(z_j)=Z_\Delta'(\til z_j)\cdot Z_{\tau^{z_1}_{\eta_1}}'(z_j)$, and by (\ref{xiH}) $|\til z_j|\ge \xi_H$ on the event $E_{\eta_1;H}$. Thus,   if $\eta_2\le \xi_H/40$, then $|Z_\Delta'(z_j)-1|\le 320  {\eta_2^2}/{\xi_H^2}$. Since $G$ is continuous on $\Sigma_{n-1}$, it is uniformly continuous on the compact set $\Omega_H$. By (\ref{Z-r-eta}),   $|G(\ulin Z_{\eta_1}(\ulin z'))-G(\ulin Z_{r_1}(\ulin z')|\to 0$ uniformly as $\eta_2\to 0^+$. Combining these facts with the compactness of $Q_H$ and the expressions of $G(\eta_1,\cdot)$ and $G(r_1,\cdot)$, we find that, there is $\delta_H''\in(0,\delta_H')$ depending only on $\kappa,n,\ulin z^*,H,\delta_H$ such that, if $\eta_2\le\delta_H''$, then
$$e_6,e_{10}< \frac\eps {11}.$$

We now explain how to choose the $H$ and $\eta_1,\eta_2$ in the approximation with errors from $e_1$ to $e_{13}$. First, we  choose the $\HH$-hull $H$ and $\delta_H>0$ such that $e_1,e_{13}\le \frac \eps{11}$ if $\Vert \ulin z-\ulin z^*\Vert_\infty\le \delta_H$ and $\Vert \ulin r\Vert_\infty\le \delta_H$. We have the quantities $C^H_1,C^H_2,C^H_3,\delta_H',\delta_H''\in(0,\infty)$ depending only on $\kappa,n,\ulin z^*,H,\delta_H$. Assume that $\Vert \ulin z-\ulin z^*\Vert_\infty\le \delta_H$. If $\Vert\ulin r\Vert_\infty\le \delta_H'$, then
 $e_3<\frac\eps {11}$. Let $\eta_2= \delta_H''$. Then we have $e_6,e_{10}<\frac\eps{11}$. Since $\delta_H''< \delta_H$, we have $\eta_2< \delta_H$, and so $e_5=e_{11}=0$. Let $\eta_1=(\eps/(11 C^H_2))^{1/\alpha} \eta_2$. Then $e_4,e_7,e_9,e_{12}\le \frac\eps{11}$.  If $r_1<(\eps/(11 C^H_1))^{1/\beta_1}$, then $e_2<\frac\eps{11}$; and if $r_1<(\eps/(11 C^H_3))^{1/\beta_2} \eta_1$, then $e_8<\frac\eps{11}$. In conclusion, if $\Vert \ulin z-\ulin z^*\Vert_\infty\le \delta_H$ and
 $$\Vert\ulin r\Vert_\infty< \delta_H'\wedge \Big(\Big(\frac{\eps}{11 C^H_3}\Big)^{1/\beta_2} \cdot\Big(\frac {\eps}{11 C^H_2}\Big)^{1/\alpha}\cdot \delta_H''\Big)=:\delta,$$
 then $e_5=e_{11}=0$ and all $e_j$'s are bounded by $\eps/11$, which then imply by (\ref{sum-ej}) that (\ref{local-uniform}) holds. Thus, we get the  existence of the limit in (\ref{ordered}) with $j_0=1$ as well as the uniform convergence on compact subsets of $\Sigma_n$.

 \subsection{Continuity}\label{proof2}
In this subsection, we prove the continuity of the function $G^1$ on $\Sigma_n$. We adopt the notation in the previous subsection.  
By the  rescaling property and left-right symmetry of SLE, for any $c\in\R\sem \{0\}$, $\ulin z=(z_1,\dots,z_n)\in \Sigma_n$, and $r>0$,
$$\PP[\tau^{ z_1}_{ r }<\tau^{z_k}_{ r }<\infty,2\le k\le n]=\PP[\tau^{c z_1}_{|c| r }<\tau^{c z_k}_{|c| r }<\infty,2\le k\le n].$$
Multiplying both sides by $r^{-n\alpha}$ and sending $r$ to $0^+$, we get by the existence of the limit in (\ref{ordered}) that
 $G^1(\ulin z)=| c|^{n \alpha} G^1( c\ulin z)$. In particular, we have $G^1(\ulin z)=|z_1|^{-n\alpha} G^1(1,z_2/z_1,\dots,z_n/z_1)$. Thus, it suffices to prove that $G^1(1,\cdot)$ is continuous on $\Sigma^1_{n-1}$, which is the set of $\ulin w\in \Sigma_{n-1}$ such that $(1,\ulin w)\in\Sigma_n$. Define $\ha G$ on $\Sigma^1_{n-1}$ such that $\ha G(\ulin w)=\EE^*_1 [G_{T_1}(\ulin w)]$. Then $G^1(1,\ulin w)=G(1) \ha G(\ulin w)$. So it suffices to prove that $\ha G$ is continuous on $\Sigma^1_{n-1}$. From the previous subsection, $\ha G$ vanishes on the set of $\ulin w$ which has at least one component lying in $(0,1)$. Since such set is open in $\Sigma^1_{n-1}$, it suffices to prove the continuity of $\ha G$ at other points of $\Sigma^1_{n-1}$.

 Fix $\ulin w^*=(w_2^*,\dots,w_n^*)\in \Sigma^1_{n-1}$ such that $w_k^*\not\in(0,1)$ for any $2\le k\le n$. Let $w_0=0$ and $w_1=1$. Let $d^*=\min\{|w_j^*-w_k^*|:0\le j<k\le n\}>0$. Let $\eps>0$. By the argument of an upper bound of $e_{13}$ in the previous subsection, there are $\delta_H\in(0,d^*/3)$ and an $\HH$-hull $H$,   such that $\{z\in\HH:|z-1|\le 3\delta_H\}\subset H$, $\dist(w_j^*,H)\ge 3\delta_H$, $2\le j\le n$, and
  for any $\ulin w=(w_2,\dots,w_n)\in\R^{n-1}$ satisfying $\Vert \ulin w-\ulin w^*\Vert_\infty\le \delta_H$, we have
$ \EE_{1}^{*}[{\mathbf 1}_{  E_{0;H}^c}  G_{T_1} (\ulin w)]<\eps/3$, where $E_{0;H}$ is the event that $\gamma[0,T_1]\subset H$. Suppose $\Vert \ulin w-\ulin w^*\Vert_\infty\le \delta_H$, we use the following approximation relations for such $H$:
$$\ha G(\ulin w)=\EE_{1}^{*}[ G _{T_1}(\ulin w)]\st{e_1}{\approx} \EE_{1}^{*}[{\mathbf 1}_{  E_{H}}  G _{T_1}(\ulin w)]\st{e_2}{\approx} \EE_{1}^{*}[{\mathbf 1}_{  E_{H}}  G_{T_1} (\ulin w^*)]\st{e_3}{\approx} \EE_{1}^{*}[   G_{T_1} (\ulin w^*)] =\ha G(\ulin w^*).$$
We have known that $e_1,e_3<\eps/3$. It remains to bound $e_2$.

We write $\ulin Z(\ulin w)=(Z_{T_1}(w_2),\dots,Z_{T_1}(w_n))$. Then   $G_{T_1}(\ulin w)=\prod_{j=2}^n Z_{T_1}'(w_j)^\alpha G(\ulin Z (\ulin w))$.
As $\ulin w\to \ulin w^*$, we have $G(\ulin Z (\ulin w))\to G(\ulin Z (\ulin w^*))$ by the continuity of $(n-1)$-point Green's function, and $Z_{T_1}'(w_j)\to Z_{T_1}'(w_j^*)$, $2\le j\le n$, which together imply that $G_{T_1}(\ulin w)\to G_{T_1}(\ulin w^*)$. We now show that the convergence is uniform (independent of the randomness) on the event $E_{0;H}$. By the previous subsection, on the event $E_{0;H}$, we have  $ \ulin Z(\ulin w),\ulin Z(\ulin w^*)\in \Omega_H$, and $Z_{T_1}'(w_j),Z_{T_1}'(w_j^*)\in Q_H$, $2\le j\le n$. By the compactness of $Q_H$, on the event $E_{0;H}$, the random map $Z_{T_1}$ is equicontinuous (independent of the randomness) on $[w_j-\delta_H,w_j+\delta_H]$ for $2\le j\le n$. Thus, as $\ulin w\to \ulin w^*$, $\ulin Z(\ulin w)\to \ulin Z(\ulin w^*)$  uniformly on the event $E_{0;H}$. Since $G$ is uniformly continuous on the compact set $\Omega_H$, we get $G(\ulin Z(\ulin w))\to G (\ulin Z(\ulin w^*))$ uniformly on the event $E_{0;H}$ as $\ulin w\to \ulin w^*$. By Koebe's distortion theorem, for $2\le j\le n$, $Z_{T_1}'(w_j)\to Z_{T_1}'(w_j^*)$ uniformly on the event $E_{0;H}$ as $\ulin w\to \ulin w^*$. Thus, $G_{T_1}(\ulin w)\to G_{T_1}(\ulin w^*)$ uniformly on the event $E_{0;H}$ as $\ulin w\to \ulin w^*$. In particular, there is $\delta_H'\in(0,\delta_H)$ such that if $\Vert \ulin w-\ulin w^*\Vert_\infty\le \delta_H'$, then $|G_{T_1}(\ulin w)-G_{T_1}(\ulin w^*)|<\eps/3$ on the event $E_{0;H}$, which implies that $e_2<\eps/3$. Thus, if $\Vert \ulin w-\ulin w^*\Vert_\infty\le \delta_H'$, then
$$|\ha G(\ulin w)-\ha G(\ulin w^*)|\le e_1+e_2+e_3<\frac \eps 3+\frac \eps 3+\frac \eps 3=\eps.$$
So we get the desired continuity of $\ha G$ at $\ulin w^*$. The proof of the continuity of $G^1$ on $\Sigma^n$ is thus complete, and so is the proof of Theorem \ref{main}.



\end{document}